\documentclass[11pt]{amsart}
\usepackage{amsopn}
\usepackage{amssymb, amscd}
\usepackage{multirow}
\usepackage{graphicx, graphics, epsfig}
\usepackage{faktor}
\usepackage{enumerate}
\usepackage{tikz}
\usepackage{pgfplots}
\usepackage{hyperref}
\usepackage{color}

\newcommand{\nc}{\newcommand}

\nc{\fg}{\mathfrak{f} } \nc{\vg}{\mathfrak{v} } \nc{\wg}{\mathfrak{w} }
\nc{\zg}{\mathfrak{z} } \nc{\ngo}{\mathfrak{n} } \nc{\kg}{\mathfrak{k} }
\nc{\mg}{\mathfrak{m} } \nc{\bg}{\mathfrak{b} } \nc{\ggo}{\mathfrak{g} } \nc{\eg}{\mathfrak{e} }
\nc{\ggob}{\overline{\mathfrak{g}} } \nc{\sog}{\mathfrak{so} }
\nc{\sug}{\mathfrak{su} } \nc{\spg}{\mathfrak{sp} } \nc{\slg}{\mathfrak{sl} }
\nc{\glg}{\mathfrak{gl} } \nc{\cg}{\mathfrak{c} } \nc{\rg}{\mathfrak{r} }
\nc{\hg}{\mathfrak{h} } \nc{\tg}{\mathfrak{t} } \nc{\ug}{\mathfrak{u} }
\nc{\dg}{\mathfrak{d} } \nc{\ag}{\mathfrak{a} } \nc{\pg}{\mathfrak{p} }
\nc{\sg}{\mathfrak{s} } \nc{\affg}{\mathfrak{aff} } \nc{\qg}{\mathfrak{q} } \nc{\lgo}{\mathfrak{l} }

\nc{\pca}{\mathcal{P}} \nc{\nca}{\mathcal{N}} \nc{\lca}{\mathcal{L}}
\nc{\oca}{\mathcal{O}} \nc{\mca}{\mathcal{M}} \nc{\tca}{\mathcal{T}}
\nc{\aca}{\mathcal{A}} \nc{\cca}{\mathcal{C}} \nc{\gca}{\mathcal{G}}
\nc{\sca}{\mathcal{S}} \nc{\hca}{\mathcal{H}} \nc{\bca}{\mathcal{B}}
\nc{\dca}{\mathcal{D}} \nc{\eca}{\mathcal{E}} \nc{\wca}{\mathcal{W}}

\nc{\vp}{\varphi} \nc{\ddt}{\tfrac{d}{dt}} \nc{\dsdt}{\tfrac{d^2}{dt^2}} \nc{\dds}{\frac{d}{ds}}
\nc{\dpar}{\frac{\partial}{\partial t}} \nc{\im}{\mathrm{i}}

\nc{\SO}{\mathrm{SO}} \nc{\Spe}{\mathrm{Sp}} \nc{\Sl}{\mathrm{SL}}
\nc{\SU}{\mathrm{SU}} \nc{\Or}{\mathrm{O}} \nc{\U}{\mathrm{U}} \nc{\Gl}{\mathrm{GL}}
\nc{\Se}{\mathrm{S}} \nc{\Cl}{\mathrm{Cl}} \nc{\Spin}{\mathrm{Spin}}
\nc{\Pin}{\mathrm{Pin}} \nc{\G}{\mathrm{GL}_n(\RR)} \nc{\g}{\mathfrak{gl}_n(\RR)}

\nc{\RR}{{\Bbb R}} \nc{\HH}{{\Bbb H}} \nc{\CC}{{\Bbb C}} \nc{\ZZ}{{\Bbb Z}}
\nc{\FF}{{\Bbb F}} \nc{\NN}{{\Bbb N}} \nc{\QQ}{{\Bbb Q}} \nc{\PP}{{\Bbb P}} \nc{\OO}{{\Bbb O}}

\nc{\vs}{\vspace{.2cm}} \nc{\vsp}{\vspace{1cm}} \nc{\ip}{\langle\cdot,\cdot\rangle}
\nc{\ipp}{(\cdot,\cdot)} \nc{\la}{\langle} \nc{\ra}{\rangle} \nc{\unm}{\tfrac{1}{2}}
\nc{\unc}{\tfrac{1}{4}} \nc{\und}{\frac{1}{16}} \nc{\no}{\vs\noindent}
\nc{\lam}{\Lambda^2(\RR^n)^*\otimes\RR^n} \nc{\tangz}{{\rm T}^{\rm Zar}}
\nc{\nor}{{\sf n}}  \nc{\mum}{/\!\!/} \nc{\kir}{/\!\!/\!\!/}
\nc{\Ri}{\tfrac{4\Ric_{\mu}}{||\mu||^2}} \nc{\ds}{\displaystyle}
\nc{\ben}{\begin{enumerate}} \nc{\een}{\end{enumerate}} \nc{\f}{\frac}
\nc{\lb}{[\cdot,\cdot]} \nc{\isn}{\tfrac{1}{||v||^2}}
\nc{\gkp}{(\ggo=\kg\oplus\pg,\ip)} \nc{\ukh}{(\ug=\kg\oplus\hg,\ip)}
\nc{\tgkp}{(\tilde{\ggo}=\kg\oplus\pg,\ip)}
\nc{\wt}{\widetilde}
\nc{\iop}{\mathtt{i}} \nc{\jop}{\mathtt{j}} 
\nc{\Hk}{H_{\kil}} \nc{\gk}{g_{\kil}}

\nc{\Hess}{\operatorname{Hess}} \nc{\ad}{\operatorname{ad}}
\nc{\Ad}{\operatorname{Ad}} \nc{\rank}{\operatorname{rk}}
\nc{\Irr}{\operatorname{Irr}} \nc{\End}{\operatorname{End}}
\nc{\Aut}{\operatorname{Aut}} \nc{\Inn}{\operatorname{Inn}}
\nc{\Der}{\operatorname{Der}} \nc{\Ker}{\operatorname{Ker}}
\nc{\Iso}{\operatorname{Iso}} \nc{\Diff}{\operatorname{Diff}}
\nc{\Lie}{\operatorname{L}} \nc{\tr}{\operatorname{tr}} \nc{\dif}{\operatorname{d}}
\nc{\sen}{\operatorname{sen}} \nc{\modu}{\operatorname{mod}}
\nc{\CRic}{\operatorname{PP}} \nc{\Cric}{\operatorname{P}} \nc{\Ricci}{\operatorname{Ric}}
\nc{\sym}{\operatorname{sym}} \nc{\herm}{\operatorname{herm}} \nc{\symac}{\operatorname{sym^{ac}}}
\nc{\symc}{\operatorname{sym^{c}}} \nc{\scalar}{\operatorname{Sc}}
\nc{\grad}{\operatorname{grad}} \nc{\ricci}{\operatorname{Rc}} \nc{\kil}{\operatorname{B}} \nc{\cas}{\operatorname{C}} \nc{\lic}{\operatorname{L}}
\nc{\Nor}{\operatorname{Norm}}  \nc{\ricc}{\operatorname{Rc^{c}}}
\nc{\Ricc}{\operatorname{Ric^{c}}} \nc{\ricac}{\operatorname{Rc^{ac}}}
\nc{\Ricac}{\operatorname{Ric^{ac}}} \nc{\Riem}{\operatorname{Rm}} \nc{\Sec}{\operatorname{Sec}}
\nc{\riccig}{\operatorname{ric^{\gamma}}} \nc{\mm}{\operatorname{m}} \nc{\Mm}{\operatorname{M}}
\nc{\Le}{\operatorname{L}} \nc{\tang}{\operatorname{T}}
\nc{\level}{\operatorname{level}} \nc{\rad}{\operatorname{r}}
\nc{\abel}{\operatorname{ab}} \nc{\CH}{\operatorname{CH}} \nc{\Cone}{{\mathcal C}} \nc{\CCone}{\operatorname{CC}} \nc{\CP}{{\mathcal P}}
\nc{\mcc}{\operatorname{mcc}} \nc{\Adj}{\operatorname{Adj}}
\nc{\Order}{\operatorname{O}}  \nc{\inj}{\operatorname{inj}} \nc{\proy}{\operatorname{pr}}
\nc{\vol}{\operatorname{vol}} \nc{\Diag}{\operatorname{Dg}} \nc{\Diagg}{\operatorname{Diag}}
\nc{\Spec}{\operatorname{Spec}} \nc{\Ima}{\operatorname{Im}} \nc{\Rea}{\operatorname{Re}}
\nc{\spann}{\operatorname{span}} \nc{\Aff}{\operatorname{Aff}} \nc{\E}{\operatorname{E}} \nc{\id}{\operatorname{id}} \nc{\dete}{\operatorname{det}} \nc{\Crit}{\operatorname{Crit}} \nc{\val}{\operatorname{val}}

\theoremstyle{plain}
\newtheorem{theorem}{Theorem}[section]
\newtheorem{proposition}[theorem]{Proposition}
\newtheorem{corollary}[theorem]{Corollary}
\newtheorem{lemma}[theorem]{Lemma}

\theoremstyle{definition}
\newtheorem{definition}[theorem]{Definition}

\theoremstyle{remark}
\newtheorem{remark}[theorem]{Remark}
\newtheorem{example}[theorem]{Example}

\title{Harmonic $3$-forms on compact homogeneous spaces}

\author{Jorge Lauret}  \author{Cynthia Will}

\address{FaMAF, Universidad Nacional de C\'ordoba and CIEM, CONICET (Argentina)}
\email{jorgelauret@unc.edu.ar} \email{cynthia.will@unc.edu.ar}

\thanks{This research was partially supported by grants from Univ. Nac. de C\'ordoba and Foncyt (Argentina)}

\date{\today}

\begin{document}

\maketitle

\begin{abstract}
The third real de Rham cohomology of compact homogeneous spaces is studied.  Given $M=G/K$ with $G$ compact semisimple, we first show that each bi-invariant symmetric bilinear form $Q$ on $\ggo$ such that $Q|_{\kg\times\kg}=0$ naturally defines a $G$-invariant closed $3$-form $H_Q$ on $M$, which plays the role of the so called Cartan $3$-form $Q([\cdot,\cdot],\cdot)$ on the compact Lie group $G$.  Indeed, every class in $H^3(G/K)$ has a unique representative $H_Q$.  Secondly, focusing on the class of homogeneous spaces with the richest third cohomology (other than Lie groups), i.e., $b_3(G/K)=s-1$ if $G$ has $s$ simple factors, we give the conditions to be fulfilled by $Q$ and a given $G$-invariant metric $g$ in order for $H_Q$ to be $g$-harmonic, in terms of algebraic invariants of $G/K$.  As an application, we obtain that any $3$-form $H_Q$ is harmonic with respect to the standard metric, although for any other normal metric, there is only one $H_Q$ up to scaling which is harmonic.  Furthermore, among a suitable $(2s-1)$-parameter family of $G$-invariant metrics, we prove that the same behavior occurs if $\kg$ is abelian: either every $H_Q$ is $g$-harmonic (this family of metrics depends on $s$ parameters) or there is a unique $g$-harmonic $3$-form $H_Q$ (up to scaling).  In the case when $\kg$ is not abelian, the special metrics for which every $H_Q$ is $g$-harmonic depend on $3$ parameters.
\end{abstract}

\tableofcontents

\section{Introduction}\label{intro}

The understanding of the de Rham cohomology $H(M)$ of a compact homogeneous space $M=G/K$ 
in terms of algebraic invariants of the groups $G$, $K$ and the embedding $K\subset G$ is a classical problem which has been studied by many renowned mathematicians, including E.\ Cartan, Chevalley, Eilenberg, Koszul, A.\ Weil, H.\ Cartan and Borel (see \cite{Brl}).  In this paper, we are interested in $H^3(M)$ over $\RR$, with a particular emphasis on finding an explicit description of $3$-forms which are harmonic with respect to a given $G$-invariant Riemannian metric.  Our interest comes from potential applications to the study of many well-known classes of geometric structures on $M$ involving closed, coclosed or harmonic $3$-forms.  

In the case when $M=G$ is a compact connected semisimple Lie group, it is well known that every class in $H^3(G)$ has a unique bi-invariant representative, which is necessarily of the form $\overline{Q}:=Q([\cdot,\cdot],\cdot)\in\Lambda^3\ggo^*$ for some bi-invariant symmetric bilinear form $Q$ on the Lie algebra $\ggo$ of $G$.  The third Betti number $b_3(G)$ is therefore the number of simple factors of $G$.  Moreover, any $\overline{Q}$ is harmonic with respect to any bi-invariant metric on $G$.  We show in \S\ref{LG-sec} that, beyond the bi-invariant context, the harmonicity of $\overline{Q}$ relative to a left-invariant metric on $G$ depends on tricky conditions in terms of the structural constants of $\ggo$.    

Let $M=G/K$ be a homogeneous space, where $G$ is a compact, connected and semisimple Lie group and $K$ is a connected closed subgroup.  It is easy to see that $b_1(G/K)=0$ and $b_2(G/K)=\dim{\zg(\kg)}$, where $\zg(\kg)$ is the center of the Lie algebra $\kg$ of $K$ (see \S\ref{preli}).  Concerning third cohomology, we found that there is also a canonical $G$-invariant closed $3$-form $H_Q$ attached to each bi-invariant symmetric bilinear form $Q$ on $\ggo$ such that $Q|_{\kg\times\kg}=0$, given by
$$
H_Q(X,Y,Z) := 4Q([X,Y],Z) - Q([X,Y]_\pg,Z) + Q([X,Z]_\pg,Y) - Q([Y,Z]_\pg,X),
$$
for all $X,Y,Z\in\pg$, where $\ggo=\kg\oplus\pg$ is any reductive decomposition (see \S\ref{H3-sec}).  An equivalent way to define $H_Q$ is by $\pi^*H_Q:=\overline{Q}+d\alpha_Q$, where $\pi:G\rightarrow G/K$ is the usual projection and $\alpha_Q\in\Lambda^2\ggo^*$ is given by 
$$
\alpha_Q|_{\kg\times\kg}=0, \qquad \alpha_Q|_{\pg\times\pg}=0, \qquad \alpha_Q(X,Z)=Q(X,Z), \qquad\forall X\in\pg,  \quad Z\in\kg. 
$$ 
Moreover, we prove that every class in $H^3(G/K)$ has a unique representative of the form $H_Q$ and so the third Betti number is given by
\begin{equation}\label{formb3}
b_3(G/K) = s - d_{G/K}, 
\end{equation}
where $d_{G/K}:=\dim\left\{ Q|_{\kg\times\kg}: Q\;\mbox{bi-invariant on}\;\ggo\right\}$ and 
$\ggo=\ggo_1\oplus\dots\oplus\ggo_s$ 
is a decomposition of $\ggo$ into simple ideals.  In particular, $b_3(G/K) = s$ if and only if $K$ is trivial.  
The authors believe that formula \eqref{formb3} must be known, although they were not able to find it in the literature.  In any case, it only takes one self-contained half page to prove it.   

Our aim in this paper is the study of the following natural problem:
\begin{quote}
Find the conditions to be fulfilled by a given bi-invariant symmetric bilinear form $Q$ (with $Q|_{\kg\times\kg}=0$) and a given $G$-invariant metric $g$ on $M=G/K$ in order for $H_Q$ to be $g$-harmonic, in terms of algebraic invariants of $G/K$.
\end{quote}

For simplicity, we consider the class of homogeneous spaces with the richest third cohomology (other than Lie groups), i.e., $b_3(G/K)=s-1$.  We also fix a decomposition $\kg=\zg(\kg)\oplus\kg_1\oplus\dots\oplus\kg_t$ in simple ideals.  We first show that such a class is characterized by the following algebraic condition on $G/K$ in the irreducible case (see Proposition \ref{al-car}), which we have called {\it aligned}: there exist $c_1,\dots,c_s>0$ such that
\begin{equation*}
\kil_{\ggo_i}(Z_i,W_i) = \tfrac{1}{c_i}\kil_\ggo(Z,W), \qquad\forall Z,W\in\kg, \quad i=1,\dots,s.
\end{equation*}   
In other words, the ideals $\kg_j$'s and the center are uniformly embedded on each $\ggo_i$ in some sense.  Here $\kil_\hg$ denotes the Killing form of a Lie algebra $\hg$.  It is easy to see that the aligned condition implies that $\kil_{\kg_j}=\lambda_j\kil_{\ggo}|_{\kg_j\times\kg_j}$ for some positive number $\lambda_j$ for all $j=1,\dots,t$ and $\pi_i(\kg)\simeq\kg$ for all $i=1,\dots,s$ (see Definition \ref{aligned} for further details).  

Note that $G/K$ is aligned and consequently $b_3(G/K)=s-1$ as soon as $\kg$ is simple (or one-dimensional) and $\pi_i(\kg)\ne 0$ for all $i=1,\dots,s$, as well as when $G=H\times\dots\times H$ ($s$-times, $s\geq 2$) and $K=\Delta L$ for any subgroup $L\subset H$ (see also Example \ref{ex3}).  In the aligned case, the condition $Q|_{\kg\times\kg}=0$ for a bi-invariant $Q$, say $Q=y_1\kil_{\ggo_1}+\dots+y_s\kil_{\ggo_s}$, is simply given by $\tfrac{y_1}{c_1}+\dots+\tfrac{y_s}{c_s}=0$.

\subsection{Main results} 
Let $M=G/K$ be an aligned homogeneous space with positive constants $c_1,\dots,c_s$ and $\lambda_1,\dots,\lambda_t$.  For any given bi-invariant metric 
$$
g_b=z_1(-\kil_{\ggo_1})+\dots+z_s(-\kil_{\ggo_s}), \qquad z_1,\dots,z_s>0,   
$$
we consider the $g_b$-orthogonal reductive decomposition $\ggo=\kg\oplus\pg$, the normal metric on $M$ determined by $g_b|_{\pg\times\pg}$ and a suitable $g_b$-orthogonal $\Ad(K)$-invariant decomposition 
$$
\pg=\pg_1\oplus\dots\oplus\pg_s\oplus\pg_{s+1}\oplus\dots\oplus\pg_{2s-1},
$$ 
see Proposition \ref{red-al} for more details.  As $\Ad(K)$-representations, $\pg_i$ is equivalent to the isotropy representation of the homogeneous space $G_i/\pi_i(K)$ for all $i=1,\dots,s$ and $\pg_j$ is equivalent to the adjoint representation $\kg$ for all $j=s+1,\dots,2s-1$ (in particular, an aligned $G/K$ is never multiplicity-free for $s\geq 3$, which makes computations much more difficult).  

We assume that none of the irreducible components of $\pg_1,\dots,\pg_s$ is equivalent to any of the simple factors of $\kg$ as $\Ad(K)$-representations and that either $\zg(\kg)=0$ or the trivial representation is not contained in any of $\pg_1,\dots,\pg_s$.  This implies that $\pg_1\oplus\dots\oplus\pg_{s}$ and $\pg_{s+1}\oplus\dots\oplus\pg_{2s-1}$ are $\omega$-orthogonal for any $\omega\in(\Lambda^2\pg^*)^K$ (see the paragraph containing \eqref{h} for more details on this assumption).  

\begin{theorem}\label{main-intro}
Given a $G$-invariant metric of the form
$$
g=x_1g_b|_{\pg_1\times\pg_1}+\dots+x_{2s-1}g_b|_{\pg_{2s-1}\times\pg_{2s-1}}, \qquad x_1,\dots,x_{2s-1}>0, 
$$ 
we set 
\begin{align*}
a_j:=&\dim{\kg}\tfrac{1}{c_{j+1}x_{j+1}^2}+Cas_0\left(\tfrac{1}{x_{s+j}^2}-\tfrac{1}{x_{j+1}^2}\right), \qquad j=1,\dots,s-1,\\ 
b_j:=&\dim{\kg}\left(\tfrac{1}{c_1x_1^2}+\dots+\tfrac{1}{c_jx_j^2}\right) + 
Cas_0\left(\tfrac{1}{x_{s+1}^2}-\tfrac{1}{x_{1}^2}+\dots+\tfrac{1}{x_{s+j}^2}-\tfrac{1}{x_{j}^2}\right), 
\end{align*}
where $Cas_0 := \lambda_1\dim{\kg_1}+\dots+\lambda_t\dim{\kg_t}$.   Then the closed $3$-form $H_Q$, where
$$
Q=y_1\kil_{\ggo_1}+\dots+y_s\kil_{\ggo_s}, \qquad \tfrac{y_1}{c_1}+\dots+\tfrac{y_s}{c_s}=0, 
$$ 
is $g$-harmonic if and only if 
\begin{equation}\label{harm-intro}
x_{s+k}\left(a_kA_k+b_k+2Cas_0\left(\tfrac{1}{x_{s+j}^2}-\tfrac{1}{x_{s+k}^2}\right)\right)C_j  
= x_{s+j}(a_jA_j+b_j)C_k, 
\end{equation}
for all $1\leq j< k\leq s-1$, where 
$$
A_j:=-\tfrac{c_{j+1}}{z_{j+1}}\left(\tfrac{z_1}{c_1}+\dots+\tfrac{z_j}{c_j}\right), \qquad C_j:=\tfrac{y_1}{c_1}+\dots+\tfrac{y_j}{c_j}+A_j\tfrac{y_{j+1}}{c_{j+1}}.  
$$
\end{theorem}

Note that condition \eqref{harm-intro} trivially holds if $s=2$ or $\kg=0$ and that for the normal metric $g=g_b$, the constants $a_j$ and $b_j$ represent just algebraic invariants of $G/K$.  We do not know whether all the $G$-invariant metrics are covered or not by the above theorem in the case when all the spaces $G_i/\pi_i(K)$ are isotropy irreducible and $K$ is either simple or one-dimensional (see Remark \ref{Hqgb-rem4}; note that otherwise it is clear that not every $G$-invariant metric is covered).  

\begin{corollary}
Let $M=G/K$ be an aligned homogeneous space and let $g$ be a $G$-invariant metric as in the above theorem.  
\begin{enumerate}[{\rm (i)}] 
\item If $\kg$ is abelian, then either every closed $3$-form $H_Q$ is $g$-harmonic (this family of metrics depends on $s$ parameters), or there is a unique $g$-harmonic $3$-form $H_Q$ (up to scaling). 

\item If $\kg$ is not abelian, then the family of metrics such that every closed $3$-form $H_Q$ is $g$-harmonic depends on $3$ parameters $x_1,x_s,x_{s+1}$ and can be described as follows: 
$$
g=(x_1,x_2,\dots,x_{s-1},x_s,x_{s+1},\dots,x_{s+1},x_{2s-1})_{g_b},
$$
where $x_2,\dots,x_{s-1}$ are determined by $x_1,x_{s+1}$ and $x_{2s-1}$ by $x_1,x_s,x_{s-1}$. 

\item The standard metric $g_{\kil}$ is the unique normal metric on $G/K$ (up to scaling) satisfying that every $H_Q$ is harmonic.  

\item For any normal metric $g_b\ne\gk$, there exists a unique $g_b$-harmonic $3$-form $H_Q$ (up to scaling).

\item For any nonzero closed $3$-form $H_Q$, there exists a one-parameter family (up to scaling) of normal metrics $g_b(t)\ne\gk$ such that $H_Q$ is $g_b(t)$-harmonic for all $t$. 
\end{enumerate}
\end{corollary}

See Remarks \ref{Hqgb-rem5}-\ref{Hqgb-rem4} and Corollaries \ref{HQgb-cor}-\ref{HQgb-cor2} for more detailed statements of the results.  

\subsection{Geometric applications}  
We have applied these results in \cite{BRF} to the existence problem of $G$-invariant generalized metrics on $M=G/K$ which are {\it Bismut Ricci flat}, i.e., precisely the fixed points of the generalized Ricci flow (see the recent book \cite{GrcStr} and \cite{PdsRff1,PdsRff2} for further information).  The structure consists of a pair $(g,H)$, where $g$ is a Riemannian metric and $H$ is a closed $3$-form, and the corresponding {\it Bismut connection} $\nabla^B$ is defined by 
$$
g(\nabla^B_XY,Z)=g(\nabla^g_XY,Z)+\unm H(X,Y,Z), \qquad\forall X,Y,Z\in\chi(M), 
$$
where $\nabla^g$ is the Levi Civita connection of $(M,g)$.  It turns out that $\nabla^B$ is Ricci flat if and only if, 
$$
\mbox{$H$ is $g$-harmonic} \quad\mbox{and}\quad \ricci(g)=\unc H_g^2, 
$$  
where $\ricci(g)$ is the Ricci tensor of $g$ and $H_g^2(X,Y):=g(\iota_XH,\iota_YH)$ for all $X,Y\in\chi(M)$.  

The results of the present paper can also be applied to a number of problems in differential geometry in the context of compact homogeneous spaces, including the study of Einstein manifolds with skew torsion (see \cite{AgrFrr, AgrFrrFrd}) and of Killing $2$-forms and $3$-forms (see \cite{Smm, BlgMrnSmm}).

\vs \noindent {\it Acknowledgements.}  We are very grateful to Manuel Amann, Christopher B\"ohm and Wolfgang Ziller for many extremely helpful conversations.  We would also like to thank the anonymous referee for his/her enormous work and really good eye.

\section{Preliminaries}\label{preli}

Let $M^n$ be a compact connected differentiable manifold.  We will assume throughout the paper that $M$ is homogeneous and fix an almost-effective transitive action of a compact connected Lie group $G$ on $M$.  The $G$-action determines a presentation $M=G/K$ of $M$ as a homogeneous space, where $K\subset G$ is the isotropy subgroup at some point $o\in M$.  We also assume that the compact Lie group $K$ is connected.  

Since $G$ is compact, the real de Rham cohomology of $M=G/K$ can be computed within $G$-invariant forms.  Given any reductive decomposition $\ggo=\kg\oplus\pg$ (i.e., $\pg$ is $\Ad(K)$-invariant), the space of all $G$-invariant $k$-forms is identified with $(\Lambda^k\pg^*)^K$, the space of $\Ad(K)$-invariant $k$-forms on $\pg$, i.e., 
$$
\alpha([Z,\cdot],\cdot,\dots,\cdot)+\dots+\alpha(\cdot,\dots,\cdot,[Z,\cdot])=0, \qquad\forall Z\in\kg, 
$$
and   
the differential $d:=d_M$ of forms on the manifold $M$ is given by $d:\Lambda^k\pg^*\rightarrow\Lambda^{k+1}\pg^*$, 
$$
d\alpha(X_1,\dots,X_{k+1}) := \sum_{i<j}(-1)^{i+j}\alpha([X_i,X_j]_\pg,X_1,\dots,\hat{X_i},\dots,\hat{X_j}\dots,X_{k+1}),
$$
giving rise to $H^k(G/K)=\Ker d/\Ima d$.  The $k$th {\it Betti number} is given by $b_k(G/K):=\dim{H^k(G/K)}$.    

Alternatively, the isomorphism 
\begin{equation}\label{deltagk}
(\Lambda^k\pg^*)^K\longrightarrow \Lambda^k(\ggo,K):=\left\{ \beta\in(\Lambda^k\ggo^*)^K:\iota_\kg\beta=0\right\}, \qquad \alpha\mapsto \hat{\alpha}:=\pi^*\alpha,
\end{equation}
where $\pi:G\rightarrow G/K$ is the usual projection map and $\iota_Z\beta:=\beta(Z,\cdot,\dots,\cdot)$, can be used to compute $H^k(G/K)$ upstairs within left-invariant forms on the Lie group $G$.  Indeed, the corresponding differential of forms on the Lie group $G$, which will be denoted by $\hat{d}$, satisfies that $\hat{d}\Lambda^k(\ggo,K)\subset\Lambda^{k+1}(\ggo,K)$ and $d\alpha=0$ if and only if $\hat{d}\hat{\alpha}=0$ (note that $\hat{d}\pi^*=\pi^*d$), so $H^k(G/K)=\Ker \hat{d}/\Ima \hat{d}$.  

Any $G$-invariant metric $g$ on $M=G/K$, which will always be identified with an $\Ad(K)$-invariant inner product on $\pg$, determines an inner product on each $\Lambda^k\pg^*$ given by
$$
g(\alpha,\beta) :=
\sum_{i_1,\dots,i_k}\alpha(X_{i_1},\dots,X_{i_k})\beta(X_{i_1},\dots,X_{i_k}), 
$$
where $\{ X_i\}$ is any $g$-orthonormal basis of $\pg$.  Note that $\{ \frac{1}{\sqrt{k!}}X^{i_1}\wedge\dots\wedge X^{i_k}\}$ is therefore a $g$-orthonormal basis of $\Lambda^k\pg^*$, where $\{ X^i\}$ is the basis of $\pg^*$ $g$-dual to $\{ X_i\}$.  If 
$$
d_g^*:(\Lambda^{k+1}\pg^*)^K\longrightarrow(\Lambda^k\pg^*)^K
$$ 
is the adjoint of $d$ with respect to $g$ (i.e., $g(d_g^*\cdot,\cdot)=g(\cdot,d\cdot)$), then a $k$-form $\alpha$ is closed and {\it coclosed} (i.e., $d_g^*g=0$) if and only if $\alpha$ is in the kernel of the Hodge Laplacian 
$$
\Delta_g:=dd_g^*+d_g^*d:(\Lambda^k\pg^*)^K\longrightarrow(\Lambda^k\pg^*)^K,
$$ 
and it is called $g$-{\it harmonic} in that case.  Since 
$$
(\Lambda^k\pg^*)^K = \rlap{$\underbrace{\phantom{\Ima d \oplus \Ker\Delta_g}}_{\Ker d}$} \Ima d \oplus\overbrace{\Ker\Delta_g \oplus \Ima d_g^*}^{\Ker d_g^*},
$$ 
we have that $H^k(G/K) \simeq \Ker \Delta_g$, that is, each class has a unique $g$-harmonic representative.  

We consider an $\Ad(K)$-invariant left-invariant metric $\hat{g}$ on $G$ such that $\hat{g}(\kg,\pg)=0$ and $\hat{g}|_{\pg\times\pg}=g$.  It is easy to check that 
$$
\hat{g}(\hat{\alpha},\hat{\beta})=g(\alpha,\beta), \qquad\forall \alpha,\beta\in(\Lambda^k\pg^*)^K,  
$$
so for any $\alpha\in(\Lambda^k\pg^*)^K$ and $\beta\in(\Lambda^{k-1}\pg^*)^K$,  
\begin{align*}
g(d_g^*\alpha,\beta) = g(\alpha,d\beta) = \hat{g}(\hat{\alpha},\pi^*d\beta) = \hat{g}(\hat{\alpha},\hat{d}\hat{\beta}) = \hat{g}(\hat{d}_{\hat{g}}^*\hat{\alpha},\hat{\beta}).   
\end{align*}
Note that $\hat{d}_{\hat{g}}^*\hat{\alpha}\in(\Lambda^{k-1}\ggo^*)^K$, but not necessarily in $\Lambda^{k-1}(\ggo,K)$.   Thus a $k$-form $\alpha\in(\Lambda^k\pg^*)^K$ is $g$-coclosed if and only if $\hat{d}_{\hat{g}}^*\hat{\alpha}$ is $\hat{g}$-orthogonal to $\Lambda^{k-1}(\ggo,K)$.   

A $G$-invariant metric is called {\it normal} when it is determined by $g_b|_{\pg\times\pg}$ for some bi-invariant metric $g_b$ on $\ggo$ and if $g_b=-\kil_{\ggo}$, where $\kil_{\ggo}$ denotes the Killing form of $\ggo$, then it is called {\it standard} and denoted by $g_{\kil}$.  We fix a normal metric $g_b$ on $M=G/K$ and consider the $g_b$-orthogonal reductive decomposition $\ggo=\kg\oplus\pg$ and the $g_b$-orthogonal decomposition
$$
\pg=\pg_0\oplus\pg_1, \qquad\mbox{where}\quad  \pg_0:=\{ X\in\pg:[\kg,X]=0\},
$$
that is, $\pg_0$ is the trivial $K$-representation isotypic component of the isotropy representation of $G/K$.

\subsection{Harmonic $1$-forms}\label{1f}
Any $1$-form on $\pg$ is of the form $\theta_X:=g_b(\cdot,X)$ for some $X\in\pg$ and since 
$$
\theta_X([Z,\cdot]) = g_b([Z,\cdot],X) = -g_b(\cdot, [Z,X]), \qquad\forall Z\in\kg,
$$
$\theta_X$ is $\Ad(K)$-invariant if and only if $X\in\pg_0$, so $(\pg^*)^K=\pg_0^*$.  It follows from   
$d\theta_X = -g_b([\cdot,\cdot],X)$ that $\theta_X$ ($X\in\pg_0$) is closed if and only if $X$ is $g_b$-orthogonal to $[\pg,\pg]$.  This is equivalent to $X\in\zg(\ggo)$, where $\zg(\ggo)$ is the center of $\ggo$ (indeed, $[\ggo,\ggo]=[\kg,\kg]+[\kg,\pg]+[\pg,\pg]$ and $X$ is already $g_b$-orthogonal to both $\kg$ and $[\kg,\pg]=\pg_1$).  We therefore obtain that 
$$
H^1(G/K) = \{ [\theta_X]: X\in\zg(\ggo)\cap\pg_0\}, \qquad b_1(G/K)=\dim{\zg(\kg)\cap\pg_0}.
$$
In particular, $b_1(G/K)=0$ if $G$ is semisimple.  Note that $[\theta_X]=\{ \theta_X\}$ and so $\theta_X$ is $g$-harmonic for all $X\in\zg(\ggo)\cap\pg_0$ and any $G$-invariant metric $g$.

\subsection{Harmonic $2$-forms}\label{2f}
Assume that $G$ is semisimple.  Any closed $2$-form on $G$ is exact and so it is necessarily of the form 
$$
\hat{\omega}_X:= g_b([\cdot,\cdot],X), \qquad\mbox{for some}\quad X\in\ggo.  
$$
Using that for all $Z\in\kg$ and $Y,W\in\ggo$,
$$
\hat{\omega}_X([Z,Y],W)+\hat{\omega}_X(Y,[Z,W]) = g_b([Z,[Y,W]],X) = -g_b([Y,W],[Z,X]),
$$
and 
$$
(\iota_Z\hat{\omega}_X)(Y) = \hat{\omega}_X(Z,Y) = g_b([Z,Y],X) = -g_b(Y,[Z,X]), 
$$
we obtain that $\hat{\omega}_X$ is $\Ad(K)$-invariant if and only if $\iota_\kg\hat{\omega}_X=0$, if and only if $[\kg,X]=0$, so
$$
\{\hat{\omega}\in\Lambda^2(\ggo,K):\hat{d}\hat{\omega}=0\} = \{ \hat{\omega}_X:X\in\zg(\kg)\oplus\pg_0\}.  
$$
This implies that
$$
\{\omega\in(\Lambda^2\pg^*)^K:d\omega=0\} = \{ \omega_X:X\in\zg(\kg)\oplus\pg_0\}, \qquad \mbox{where}\quad \omega_X:=g_b([\cdot,\cdot]_\pg,X),
$$
and since $\omega_X=-d\theta_X$ for any $X\in\pg_0$, we obtain that 
$$
H^2(G/K) = \{ [\omega_Z]: Z\in\zg(\kg)\}\simeq\zg(\kg), \qquad b_2(G/K)=\dim{\zg(\kg)}.
$$
Note that $[\omega_Z]=\{\omega_{Z+X}:X\in\pg_0\}$ for any $Z\in\zg(\kg)$, so if $\pg_0=0$ then $\omega_Z$ is $g$-harmonic for every $G$-invariant metric $g$.  

If $\{ e_i\}$ is a $g_b$-orthonormal basis of $\pg$, then for any $Z\in\zg(\kg)\oplus\pg_0$ and $X\in\pg_0$, 
\begin{align*}
g_b(\omega_Z,d\theta_X) =& \sum_{i,j} \omega_Z(e_i,e_j)d\theta_X(e_i,e_j) = -\sum_{i,j} g_b([e_i,e_j],Z)g_b([e_i,e_j],X) \\ 
=& -\tr{\ad{Z}|_\pg\ad{X}|_\pg} = -\kil_\ggo(Z,X).   
\end{align*}
Thus the closed $2$-form $\omega_Z$ is $g_b$-harmonic if and only if $\kil_\ggo(Z,\pg_0)=0$.  The $g_b$-harmonic representative inside a class $[\omega_Z]$, $Z\in\zg(\kg)$ is therefore given by $\omega_{Z+X_Z}$, where $X_Z$ is the unique vector in $\pg_0$ such that $\kil_\ggo(X,X_Z)=-\kil_\ggo(X,Z)$ for all $X\in\pg_0$.  Note that $\pg_0$ and so $X_Z$ depend on $g_b$.  

In particular, $\omega_Z$ is $\gk$-harmonic for any $Z\in\zg(\kg)$.

\section{Harmonic $3$-forms on compact Lie groups}\label{LG-sec}

Let $M^n=G$ be a compact connected Lie group.  It is well known that every class in $H^3(G)$ has a unique bi-invariant representative (see \cite[Chapter V, Corollary 12.7]{Brd}), which in the semisimple case, is necessarily of the form 
$$
\overline{Q}(X,Y,Z):=Q([X,Y],Z), \qquad \forall X,Y,Z\in\ggo, 
$$ 
for some bi-invariant symmetric bilinear form $Q$ on $\ggo$ (see \cite[Chapter V, Theorem 12.10]{Brd}).  These closed $3$-forms are often called {\it Cartan $3$-forms}.  Thus the third Betti number $b_3(G):=\dim{H^3(G)}$ is precisely the number of simple factors if $G$ is semisimple, and it is well known that $H^1(G)=0$ and $H^2(G)=0$ in that case (see \S\ref{1f} and \S\ref{2f}).  Note that if $G$ is simple, then $H^3(G)=\RR[\Hk]$, where $\Hk:=\overline{\kil_\ggo}$ and $\kil_\ggo$ is the Killing form of $\ggo$.  It is also well known that Cartan $3$-forms are all harmonic with respect to any bi-invariant metric on $G$ (see e.g.\ \cite[Theorem 3.4.7]{Vss}).  We study in this section the harmonicity condition for a Cartan $3$-form with respect to any left-invariant metric on $G$.       

We fix from now on a bi-invariant metric $g_b$ on $G$.  For any left-invariant metric $g$ on $G$, there exists a $g_b$-orthonormal basis $\{ e_1,\dots,e_n\}$ of $\ggo$ such that $g(e_i,e_j)=x_i\delta_{ij}$ for some $x_1,\dots,x_n>0$, which will be denoted by $g=(x_1,\dots,x_n)_{g_b}$.  Note that $\{ e_i/\sqrt{x_i}\}$ is a $g$-orthonormal basis of $\ggo$ with dual basis $\{ \sqrt{x_i}e_i\}$, where $e_i$ also denotes the dual basis defined by $e_i(e_j):=\delta_{ij}$.  The ordered basis $\{ e_1,\dots,e_n\}$ determines structural constants given by 
$$
[e_i,e_j]=\sum_k c_{ij}^ke_k, \qquad \mbox{or equivalently}, \quad c_{ij}^k:=g_b([e_i,e_j],e_k).      
$$

\begin{lemma}\label{dgstar}
For any metric $g=(x_1,\dots,x_n)_{g_b}$ and $\beta\in\Lambda^3\ggo^*$, 
$$
d_g^*\beta = -\frac{3}{2}\sum_{k<l} \left(x_k\sum_{i,j} \frac{c_{ij}^k\beta(e_i,e_j,e_l)}{x_ix_j} -x_l\sum_{i,j} \frac{c_{ij}^l\beta(e_i,e_j,e_k)}{x_ix_j}\right)e_k\wedge e_l.
$$
\end{lemma}

\begin{proof}
For any $2$-form $\alpha$,   
\begin{align*}
g(d_g^*\beta,\alpha) =& g(\beta,d\alpha) = \sum_{i,j,k} \frac{\beta(e_i,e_j,e_k)d\alpha(e_i,e_j,e_k)}{x_ix_jx_k} \\ 
=&  \sum_{i,j,k} \frac{\beta(e_i,e_j,e_k) \left(-\alpha([e_i,e_j],e_k) +\alpha([e_i,e_k],e_j) -\alpha([e_j,e_k],e_i)\right)}{x_ix_jx_k} \\ 
=&  -3\sum_{i,j,k} \frac{\beta(e_i,e_j,e_k) \alpha([e_i,e_j],e_k)}{x_ix_jx_k} 
= -3\sum_{i,j,k,l} \frac{\beta(e_i,e_j,e_k)c_{ij}^l \alpha(e_l,e_k)}{x_ix_jx_k},
\end{align*}
so for $\alpha=e_r\wedge e_s$, one obtains that
$$
g(d_g^*\beta,\alpha) = -3\sum_{i,j} \frac{\beta(e_i,e_j,e_s)c_{ij}^r}{x_ix_jx_s} - \frac{\beta(e_i,e_j,e_r)c_{ij}^s}{x_ix_jx_r}. 
$$
The fact that $\left\{\sqrt{\frac{x_kx_l}{2}}e_k\wedge e_l\right\}$ is a $g$-orthonormal basis of $\Lambda^2\ggo^*$ concludes the proof.  
\end{proof}

We fix a decomposition 
\begin{equation}\label{dec-sim}
\ggo=\ggo_0\oplus\ggo_1\oplus\dots\oplus\ggo_s, 
\end{equation}
where the $\ggo_i$'s are simple ideals of $\ggo$ and $\ggo_0$ is the center of $\ggo$.  Thus 
$$
g_b=g_0+z_1(-\kil_{\ggo_1})+\dots+z_s(-\kil_{\ggo_s}), \qquad\mbox{for some}\quad z_1,\dots,z_s>0,
$$
and some inner product $g_0$ on $\ggo_0$.  Since the set of Ricci eigenvalues of $g_b$ is $\{\unc z_1^{-1},\dots,\unc z_s^{-1}\}$, the moduli space of all bi-invariant metrics on $G$ up to isometry and scaling depends on $s-1$ parameters.  Note that $g_b$ is Einstein if and only if $\ggo_0=0$ (i.e., $G$ semisimple) and $z_1=\dots=z_s$, that is, a single point in the moduli space.        

The following corollary of Lemma \ref{dgstar} shows that in general, for a given left-invariant metric $g$, even the $g$-harmonicity of the Cartan $3$-form $\Hk$ depends on tricky conditions in terms of the structural constants.  

\begin{corollary}\label{Hkgharm-simple}
For any metric $g=(x_1,\dots,x_n)_{g_b}$ as above, the following holds.
\begin{enumerate}[{\rm (i)}] 
\item 
$$
d_g^*\Hk = -\frac{3}{2}\sum_{\substack{1\leq k<l\leq n}} (x_k-x_l)\left(\sum_{1\leq i,j\leq n} \frac{c_{ij}^kc_{ij}^l}{x_ix_j}\right)e_k\wedge e_l. 
$$
\item $\Hk$ is $g$-harmonic if and only if,
$$
\sum_{1\leq i,j\leq n} \frac{c_{ij}^kc_{ij}^l}{x_ix_j} = 0, \qquad\forall k,l \quad\mbox{such that}\quad x_k\ne x_l.
$$
In particular, $\Hk$ is $g_b$-harmonic for any bi-invariant metric $g_b$. 

\item $\Hk+d\alpha$, where $\alpha\in\Lambda^2\ggo^*$, is $g$-harmonic if and only if 
\begin{align*}
(x_k-x_l)\sum_{1\leq i,j\leq n} \frac{c_{ij}^kc_{ij}^l}{x_ix_j} =& -x_k\sum_{1\leq i,j\leq n} \frac{c_{ij}^kd\alpha(e_i,e_j,e_l)}{x_ix_j} \\
&+ x_l\sum_{1\leq i,j\leq n} \frac{c_{ij}^ld\alpha(e_i,e_j,e_k)}{x_ix_j},  
\end{align*}
for all $1\leq k<l\leq n$.
\end{enumerate}
\end{corollary}

\begin{example}
It is easy to see that the usual basis $\{ e_{rs}:= E_{rs}-E_{sr}\}$ of $\sog(n)$ satisfies that $c_{ij}^kc_{ij}^l=0$ for all $k\ne l$.  Thus $\Hk$ is $g$-harmonic for any metric $g$ on $\SO(n)$ which is diagonal with respect to $\{ e_{rs}\}$ (i.e., $\{ e_{rs}\}$ is $g$-orthogonal) by Corollary \ref{Hkgharm}.  A basis of a Lie algebra is called {\it nice} when it satisfies the above condition together with $c_{ij}^kc_{rs}^k=0$ as soon as $\{ i,j\}\cap\{ r,s\}\ne\emptyset$ (which follows from the former one if the basis is orthogonal with respect to a bi-invariant metric).  On a nilpotent Lie group, a basis satisfies that any metric which is diagonal has also a diagonal Ricci tensor if and only if it is nice (see \cite{nicebasis}), and the same holds for bases which are orthogonal with respect to a bi-invariant metric on compact Lie groups (see \cite{Krs}).  However, the above is the only nice basis known so far on a compact simple Lie group, their existence on the other compact simple Lie groups is still open.      
\end{example}

\begin{example}
For the standard $\gk$-orthogonal basis $\{ e_1,\dots,e_8\}$ of $\sug(3)$, the only nonzero products of the form $c_{ij}^kc_{ij}^l$ are $c_{36}^1c_{36}^2$ and $c_{47}^1c_{47}^2$, where $c_{36}^1=\frac{\sqrt{3}}{6}=c_{47}^1$ and $c_{36}^2=-\unm=-c_{47}^2$.  Thus $\Hk$ is $g$-harmonic for a diagonal metric $g=(x_1,\dots,x_8)_{\gk}$ on $\SU(3)$ if and only if $x_1=x_2$ or $x_3x_6=x_4x_7$ (see Corollary \ref{Hkgharm-simple}).  Moreover, using that the only nonzero brackets involving $e_1$ or $e_2$ are $[e_1,e_2]=[e_2,e_5]=[e_2,e_8]=0$ and
$$
[e_3,e_6]=\tfrac{\sqrt{3}}{6}e_1-\tfrac{1}{2} e_2, \quad [e_4,e_7]=\tfrac{\sqrt{3}}{6}e_1+\tfrac{1}{2} e_2, \quad [e_5,e_8]=\tfrac{\sqrt{3}}{3}e_1,
$$
it is straightforward to see that the $g$-harmonic $3$-form is given by $\Hk+td(e_1\wedge e_2)$, where
$$
t:=\frac{-\sqrt{3}(x_1-x_2)\left(\frac{1}{x_3x_6}-\frac{1}{x_4x_7}\right)}{(x_1+3 x_2)\left(\frac{1}{x_3x_6}+\frac{1}{x_4x_7}\right)+12\frac{x_1}{x_5x_8}}.
$$
\end{example}

Any bi-invariant symmetric bilinear form is of the form 
$$
Q=Q_0+y_1\kil_{\ggo_1}+\dots+y_s\kil_{\ggo_s}, \qquad\mbox{for some}\quad y_1,\dots,y_s\in\RR,
$$
where $Q_0$ is any symmetric bilinear form on $\ggo_0$.  

We assume that the $g_b$-orthonormal basis $\{ e_i\}$ considered above is adapted to decomposition \eqref{dec-sim}, in the sense that it is the union of bases $\{ e^i_\alpha:\alpha=1,\dots,n_i\}$ of each $\ggo_i$, $i=0,1,\dots,s$, where $n_i:=\dim{\ggo_i}$, and we denote by $c_{i\alpha,i\beta}^{i\gamma}:=g_b([e^i_\alpha,e^i_\beta],e^i_\gamma)$ the corresponding structural constants.  Any metric $g$ such that decomposition \eqref{dec-sim} is $g$-orthogonal can be written on each $\ggo_i$ as $(x_1^i,\dots,x_{n_i}^i)_{g_b}$.  Note that $g$ is bi-invariant if and only if $x_1^i=\dots=x_{n_i}^i$ for all $i=1\dots,s$.  

The following corollary of Lemma \ref{dgstar} implies the well-known fact that any bi-invariant $3$-form on a compact Lie group is harmonic with respect to any bi-invariant left-invariant metric.  

\begin{corollary}\label{Hkgharm}
For any $g_b$, $Q$ and $g$ as above (assume that decomposition \eqref{dec-sim} is $g$-orthogonal), the following holds.
\begin{enumerate}[{\rm (i)}] 
\item 
$$
d_g^*\overline{Q} = -\frac{3}{2}\sum_{\substack{1\leq \gamma<\delta\leq n_i\\ 1\leq i\leq s}} \frac{y_i}{z_i}(x_\gamma^i-x_\delta^i)\left(\sum_{1\leq \alpha,\beta\leq n_i} \frac{c_{i\alpha,i\beta}^{i\gamma}c_{i\alpha,i\beta}^{i\delta}}{x_\alpha^ix_\beta^i}\right)e^i_\gamma\wedge e^i_\delta. 
$$
\item $\overline{Q}$ is $g$-harmonic if and only if for all $i=1,\dots,s$,
$$
\sum_{1\leq \alpha,\beta\leq n_i} \frac{c_{i\alpha,i\beta}^{i\gamma}c_{i\alpha,i\beta}^{i\delta}}{x_\alpha^ix_\beta^i} = 0, \qquad\forall \gamma,\delta \quad\mbox{such that}\quad x_\gamma^i\ne x_\delta^i.
$$
In particular, $\overline{Q}$ is $g$-harmonic for any bi-invariant metric $g_b$. 

\item $\overline{Q}+d\omega$, where $\omega\in\Lambda^2\ggo^*$, is $g$-harmonic if and only if 
\begin{align*}
(x_\gamma^i-x_\delta^i)\sum_{1\leq \alpha,\beta\leq n_i} \frac{c_{i\alpha,i\beta}^{i\gamma}c_{i\alpha,i\beta}^{i\delta}}{x_\alpha^ix_\beta^i} =& -x_\gamma^i\sum_{1\leq \alpha,\beta\leq n_i} \frac{c_{i\alpha,i\beta}^{i\gamma}d\omega(e_\alpha^i,e_\beta^i,e_\delta^i)}{x_\alpha^ix_\beta^i} \\
&+ x_\delta^i\sum_{1\leq \alpha,\beta\leq n_i} \frac{c_{i\alpha,i\beta}^{i\delta}d\omega(e_\alpha^i,e_\beta^i,e_\gamma^i)}{x_\alpha^ix_\beta^i},  
\end{align*}
for all $1\leq \gamma<\delta\leq n_i$ and $i=1,\dots,s$.
\end{enumerate}
\end{corollary}

\section{Third cohomology of compact homogeneous spaces}\label{H3-sec}

In this section, we study the real de Rham cohomology of $3$-forms on a homogeneous space $M=G/K$ as in \S\ref{preli} with $G$ semisimple.  As in the Lie group case studied in \S\ref{LG-sec}, each bi-invariant symmetric bilinear form $Q$ on $\ggo$ defines a $G$-invariant $3$-form on $G/K$ given by $\widetilde{Q}:=Q([\cdot,\cdot],\cdot)\in(\Lambda^3\pg^*)^K$.  However, $\widetilde{Q}$ is never closed if nonzero (it is however $g_b$-coclosed with respect to any normal metric $g_b$ on $G/K$, see \eqref{Qbcc} below) and so the following question arises:  
\begin{quote}
Does any bi-invariant symmetric bilinear form $Q$ on $\ggo$ naturally determine a $G$-invariant closed $3$-form on $G/K$?
\end{quote}

We consider fixed decompositions, 
\begin{equation}\label{decs}
\ggo=\ggo_1\oplus\dots\oplus\ggo_s, \qquad \kg=\kg_0\oplus\kg_1\oplus\dots\oplus\kg_t, 
\end{equation}
where the $\ggo_i$'s and $\kg_j$'s are simple ideals of $\ggo$ and $\kg$, respectively, and $\kg_0$ is the center of $\kg$ of dimension $d_0$.  The following result answers the above question in a satisfactory way.    

\begin{proposition}\label{H3}
Let $M=G/K$ be a homogeneous space as in \S\ref{preli} with $G$ semisimple and consider any reductive decomposition $\ggo=\kg\oplus\pg$.  
\begin{enumerate}[{\rm (i)}]
\item Each bi-invariant symmetric bilinear form $Q$ on $\ggo$ such that $Q|_{\kg\times\kg}=0$ defines a closed $3$-form $H_Q\in(\Lambda^3\pg^*)^K$ by 
\begin{align}
H_Q(X,Y,Z) :=& 4Q([X,Y],Z) - Q([X,Y]_\pg,Z) + Q([X,Z]_\pg,Y) - Q([Y,Z]_\pg,X) \label{HQ}\\ 
=& Q([X,Y],Z) + Q([X,Y]_\kg,Z) - Q([X,Z]_\kg,Y) + Q([Y,Z]_\kg,X), \notag
\end{align}
for all $X,Y,Z\in\pg$.  Equivalently, $\hat{H_Q}:=\overline{Q}+\hat{d}\alpha_Q$, where 
$\alpha_Q\in\Lambda^2\ggo^*$ is defined by 
$$
\alpha_Q|_{\kg\times\kg}=0, \qquad \alpha_Q|_{\pg\times\pg}=0, \qquad \alpha_Q(X,Z)=Q(X,Z), \qquad\forall X\in\pg,  \quad Z\in\kg. 
$$ 
\item Every class in $H^3(G/K)$ has a unique representative of the form $H_Q$ as in part (i), that is, 
$$
H^3(G/K)\simeq\left\{ Q\in\sym^2(\ggo)^G:Q|_{\kg\times\kg}=0\right\} \qquad \mbox{and}\qquad b_3(G/K)=s-d_{G/K},
$$
where   
$$
d_{G/K}:=\dim\left\{ Q|_{\kg\times\kg}: Q\;\mbox{is a bi-invariant symmetric bilinear form on}\; \ggo\right\}.
$$  
\end{enumerate}
\end{proposition}

\begin{remark}
The definition of the $3$-form $H_Q$ given in \eqref{HQ} depends on the reductive complement $\pg$.  
\end{remark}

\begin{proof}
Given a closed $3$-form $H\in(\Lambda^3\pg^*)^K$, there exist a bi-invariant symmetric bilinear form $Q$ on $\ggo$ and $\alpha\in(\Lambda^2\ggo^*)^K$ such that $\hat{H}=\overline{Q}+\hat{d}\alpha$.  Thus
$$
Q([\cdot,\cdot],Z) = -\iota_Z\hat{d}\alpha = \hat{d}\iota_Z\alpha-\lca_Z\alpha = \alpha([\cdot,\cdot],Z), \qquad\forall Z\in\kg,
$$
which implies that $Q(X,Z)=\alpha(X,Z)$ for all $X\in\ggo$, $Z\in\kg$.  In particular, $Q|_{\kg\times\kg}=0$ and for all $X,Y,W\in\pg$, 
\begin{align*}
H(X,Y,W) =& \hat{H}(X,Y,W) = Q([X,Y],W)+\hat{d}\alpha(X,Y,W) \\ 
=& Q([X,Y],W)-\alpha([X,Y],W) +\alpha([X,W],Y)-\alpha([Y,W],X) \\ 
=& Q([X,Y],W)-\alpha([X,Y]_\kg,W) +\alpha([X,W]_\kg,Y) \\ 
&-\alpha([Y,W]_\kg,X) 
 -\alpha([X,Y]_\pg,W) +\alpha([X,W]_\pg,Y)-\alpha([Y,W]_\pg,X) \\ 
 =& Q([X,Y],W) + Q([X,Y]_\kg,W) -Q([X,W]_\kg,Y) +Q([Y,W]_\kg,X) \\
 &-\alpha([X,Y]_\pg,W) +\alpha([X,W]_\pg,Y)-\alpha([Y,W]_\pg,X) \\ 
 =&H_Q(X,Y,W)+d\beta(X,Y,W),
\end{align*}
where $\beta:=\alpha|_{\pg\times\pg}\in(\Lambda^2\pg^*)^K$, that is, $H=H_Q+d\beta$.  Thus $H_Q\in(\Lambda^3\pg^*)^K$, $dH_Q=0$ and so the first statement of part (i) and part (ii) follow.  Conversely, given a bi-invariant $Q$ such that $Q|_{\kg\times\kg}=0$, we consider the $\hat{d}$-closed $3$-form $\overline{Q}+\hat{d}\alpha_Q\in\Lambda^3(\ggo,K)$.  This $3$-form therefore corresponds to a closed $3$-form in $(\Lambda^3\pg^*)^K$, which by the above computation coincides with $H_Q$, that is, $\hat{H}_Q=\overline{Q}+\hat{d}\alpha_Q$, concluding the proof of part (i).  
\end{proof}

Given $Z\in\kg$, there exist unique vectors $Z_i\in\ggo_i$ such that $Z=(Z_1,\dots,Z_s)$, where $Z_i=\pi_i(Z)$ and $\pi_i:\ggo\rightarrow\ggo_i$ are the projections relative to the decomposition $\ggo=\ggo_1\oplus\dots\oplus\ggo_s$.  Thus $\pi_i$ is a Lie algebra homomorphism and $\pi_i:\kg_j\rightarrow\pi_i(\kg_j)$ is an isomorphism as soon as $\pi_i(\kg_j)\ne 0$, in which case $\pi_i(\kg_j)$ is a simple Lie subalgebra of $\ggo_i$ for any $j\geq 1$.  For any pair $i,j$, $1\leq i\leq s$ and $0\leq j\leq t$, there exists a constant $0\leq c_{ij}\leq 1$ such that
\begin{equation}\label{kilcij}
\kil_{\pi_i(\kg_j)} = c_{ij}\kil_{\ggo_i}|_{\pi_i(\kg_j)\times\pi_i(\kg_j)},  
\end{equation}
where $\kil_{\pi_i(\kg_j)}$ and $\kil_{\ggo_i}$ are respectively the Killing forms of $\pi_i(\kg_j)$ and  $\kil_{\ggo_i}$.  We have that $c_{ij}=0$ if and only if $j=0$ or $\pi_i(\kg_j)=0$, and $c_{ij}=1$ if and only if $\pi_i(\kg_j)=\ggo_i$.  Note that if $\pi_i(\kg_j)\ne 0$, then 
$$
\kil_{\kg}(Z,W) = \kil_{\kg_j}(Z,W) = \kil_{\pi_i(\kg_j)}(Z_i,W_i) = c_{ij}\kil_{\ggo_i}(Z_i,W_i), \qquad\forall Z,W\in\kg_j. 
$$
A homogeneous space $G/K$ is determined by the Lie groups $G$ and $K$ and last but not least, by the way $K$ is embedded in $G$.  Note that each {\it Killing constant} $c_{ij}$ is an invariant of how is $\kg_j$ embedded in $\ggo_i$.       

The following bounds and equalities on $b_3(G/K)$ are actually immediate consequences of the spectral sequence for the fibration $G\rightarrow G/K \rightarrow B_K$.  We prove them for self-containedness.  

\begin{proposition}\label{H3-cor}
Let $M=G/K$ be a homogeneous space as above.  
\begin{enumerate}[{\rm (i)}]
\item If $\ggo$ has $s$ simple factors and $\kg$ has $t$ simple factors and $d_0$-dimensional center, then 
$$
s-t-\tfrac{d_0(d_0+1)}{2} \leq b_3(G/K) \leq s.  
$$
\item If $K$ is semisimple, then a bi-invariant $Q=y_1\kil_{\ggo_1}+\dots+y_s\kil_{\ggo_s}$,  
$y_i\in\RR$, satisfies that $Q|_{\kg\times\kg}=0$ if and only if the following $t$ conditions hold:
\begin{equation}\label{cij}
\sum_{i:\pi_i(\kg_j)\ne 0}\frac{1}{c_{ij}}y_i=0, \qquad\forall j=1,\dots,t,  \qquad\mbox{where}\quad \kil_{\pi_i(\kg_j)} = c_{ij}\kil_{\ggo_i}|_{\pi_i(\kg_j)\times\pi_i(\kg_j)}. 
\end{equation}

\item Equality $b_3(G/K)=s$ holds if and only if $\kg=0$ (i.e., $K$ is the trivial subgroup).  

\item $b_3(G/K)=s-1$ if $\kg$ is simple or one-dimensional.  
\end{enumerate}
\end{proposition}

\begin{remark}
Generically, 
$$
b_3(G/K) = \left\{\begin{array}{ll} s-t-\tfrac{d_0(d_0+1)}{2}, \qquad & \mbox{if}\quad s\geq t+\tfrac{d_0(d_0+1)}{2},\\ \\
0, \qquad &\mbox{if}\quad  s<t+\tfrac{d_0(d_0+1)}{2}. 
\end{array}\right. 
$$
See Examples \ref{ex1}, \ref{ex2} and \ref{ex3} below for exceptions.  
\end{remark}

\begin{proof}
It follows from Proposition \ref{H3}, (ii) that $b_3(G/K)\leq s$.  We first prove part (ii).  We have that $Q(\kg_j,\kg_l)=0$ for all $j\ne l$ and it follows from \eqref{kilcij} that
$$
Q(Z,Z) = \sum_{i=1}^s y_i\kil_{\ggo_i}(Z_i,Z_i) = \sum_{i:\pi_i(\kg_j)\ne 0}\frac{y_i}{c_{ij}}\kil_{\kg}(Z,Z), \qquad\forall Z\in\kg_j, \quad j=1,\dots,t. 
$$
Thus $Q(Z,Z)=0$ for all $Z\in\kg$ if and only if condition \eqref{cij} holds.  

On the other hand, if $\{ Z^1,\dots,Z^{d_0}\}$ is a basis of $\kg_0$, then for each $Z=a_1Z^1+\dots+a_{d_0}Z^{d_0}\in\kg_0$, $Q(Z,Z)=0$ if and only if $(y_1,\dots,y_s)$ is orthogonal with respect to the canonical inner product of $\RR^s$ to   
\begin{equation}\label{Zk0}
\begin{array}{c} 
\sum\limits_{\alpha=1}^{d_0} a_\alpha^2\left(\kil_{\ggo_1}(Z^\alpha_1,Z^\alpha_1),\dots,\kil_{\ggo_s}(Z^\alpha_s,Z^\alpha_s)\right) \\ + \sum\limits_{1\leq \alpha<\beta\leq d_0} 2a_\alpha a_\beta\left(\kil_{\ggo_1}(Z^\alpha_1,Z^\beta_1),\dots,\kil_{\ggo_s}(Z^\alpha_s,Z^\beta_s)\right),
\end{array}
\end{equation}
which, by varying $Z\in\kg_0$, generates a subspace of dimension $\leq \frac{d_0(d_0+1)}{2}$.  This, together with \eqref{cij}, gives that $s-t-\tfrac{d_0(d_0+1)}{2} \leq b_3(G/K)$, so part (i) follows.  Moreover, the only way to obtain that $b_3(G/K)=s$ is that $\pi_i(\kg_j)=0$ for all $i,j$ (i.e., $[\kg,\kg]=0$) and $\kil_{\ggo_i}(Z^\alpha_i,Z^\alpha_i)=0$ for all $1\leq \alpha\leq d_0$, $i=1,\dots,s$ (i.e., $\kg_0=0$), that is, if and only if $\kg=0$, which proves (iii).  Finally, part (iv) follows from (i) and (iii), concluding the proof. 
\end{proof}

\begin{example}\label{ex1}
For $M=\SU(2n)\times\SU(2n)/\Delta\SU(2)^n$ with the standard block diagonal embedding, we have $s=2$, $t=n$, $d_0=0$ and $c_{1j}=c_{2j}=\frac{1}{n}$ for all $j=1,\dots,n$ (see \cite[pp.37]{DtrZll}).  It follows from \eqref{cij} that $Q:= \kil_{\ggo_1}-\kil_{\ggo_2}$ produces a nonzero class $[H_Q]$ and so $b_3(G/K)=1$, in spite of $s-t-\tfrac{d_0(d_0+1)}{2}=2-n$ is way negative for $n$ large.  
\end{example}

\begin{example}\label{ex2}
If $M=\SU(n)\times\SU(n)/\Delta T^{n-1}$, then $s=2$, $t=0$ and $d_0=n-1$, and one obtains from \eqref{Zk0} that $[H_Q]\ne 0$ for $Q:= \kil_{\ggo_1}-\kil_{\ggo_2}$.  Thus $b_3(M)=1$ and the number $s-t-\tfrac{d_0(d_0+1)}{2}=2-\frac{n(n-1)}{2}$ is however very negative.  
\end{example}

Conditions \eqref{cij} and \eqref{Zk0} motivate the definition of the following concept guaranteing a large third Betti number for a homogeneous space.      

\begin{definition}\label{aligned}
A homogeneous space $G/K$ as above, where $G$ has $s$ simple factors and $K$ has $t$ simple factors and $d_0$-dimensional center, is called {\it aligned} when there exist positive constants $c_1,\dots,c_s$ such that: 
\begin{enumerate}[{\rm (i)}] 
\item The Killing constants defined in \eqref{kilcij} satisfy 
$$
(c_{1j},\dots,c_{sj}) = \lambda_j(c_1,\dots,c_s), \quad\mbox{for some}\quad \lambda_j>0, \qquad\forall j=1,\dots,t.
$$
\item There exists an inner product $\ip$ on $\kg_0$ such that 
$$
\kil_{\ggo_i}(Z_i,W_i)=\frac{-1}{c_i}\la Z,W\ra, \qquad\forall Z,W\in\kg_0.   
$$  
\item $\frac{1}{c_1}+\dots+\frac{1}{c_s}=1$.  In particular, $\kil_{\kg_j}=\lambda_j\kil_{\ggo}|_{\kg_j\times\kg_j}$ for all $j=1,\dots,t$ and $\ip=-\kil_{\ggo}|_{\kg_0\times\kg_0}$. 
\end{enumerate}
\end{definition}

If $G/K$ is aligned, then $\pi_i$ is injective for all $i$ (in particular, $\pi_i(\kg_j)\simeq\kg_j$ and $\pi_i(\kg)\simeq\kg$ for all $i,j$) and the Killing constants are given by
$$
c_{ij}=c_i\lambda_j >0, \qquad\forall i=1,\dots,s, \quad j=1,\dots,t.  
$$
\begin{example}
If $\pi_i(\kg)\ne 0$ for all $i=1,\dots,s$ and $\kg$ is simple or one-dimensional, then $G/K$ is aligned.    
\end{example}

\begin{example}\label{GLO}
If $\ggo_1=\dots=\ggo_s=\hg$ and $\pi_1=\dots=\pi_s$, i.e., $G=H\times\dots\times H$ ($s$-times) and $K=\Delta L$ for some subgroup $L\subset H$, then $G/K$ is aligned with $(c_1,\dots,c_s)=(s,\dots,s)$ and $\kil_{\lgo_j}=s\lambda_j\kil_\hg|_{\lgo_j\times\lgo_j}$ (or equivalently, $\kil_{\lgo_j}=\lambda_j\kil_\ggo|_{\lgo_j\times\lgo_j}$) for each simple factor $\lgo_j$ of $\lgo$.  It is easy to see that $M=G/K$ is diffeomorphic to $(H/L)\times H^{s-1}$, where $H^{s-1}:=H\times\dots\times H$ ($(s-1)$-times).  In the particular case when $L=H$ these spaces are called {\it Ledger-Obata} (see \cite{NklNkn}).  
\end{example}

\begin{proposition}\label{al-car}
The following conditions on a homogeneous space $G/K$ as above are equivalent: 
\begin{enumerate}[{\rm (i)}]
\item $G/K=G'/K\times H$, where $G'/K$ is aligned and $H$ is semisimple. 

\item $b_3(G/K)=s-1$. 

\item There exists an inner product $\ip$ on $\kg$ such that $Q|_{\kg\times\kg}$ coincides with $\ip$ up to scaling for any bi-invariant symmetric bilinear form $Q$ on $\ggo$.  
\end{enumerate}
\end{proposition}

\begin{remark}
The proof of part (iii) assuming (i) given below determines the inner product $\ip=-\kil_{\ggo}|_{\kg\times\kg}$.  
\end{remark}

\begin{proof}
The equivalence between (ii) and (iii) follows from Proposition \ref{H3}, (ii), so we will only prove the equivalence between (i) and (iii).  If $G/K$ is aligned, then there is an {\it adapted basis} $\{ Z^1,\dots,Z^{\dim{\kg}}\}$ of $\kg=\kg_0\oplus\kg_1\oplus\dots\oplus\kg_t$ (in the sense that it is the union of bases for each $\kg_j$) such that 
\begin{align*}
\kil_{\ggo_i}(Z^\alpha_i,Z^\beta_i) =& \delta_{\alpha\beta}\frac{-1}{c_i}, \qquad\forall Z^\alpha,Z^\beta\in\kg_0, \quad i=1,\dots,s, \quad \mbox{(see Definition \ref{aligned}, (ii))},\\  
\kil_{\kg_j}(Z^\alpha,Z^\beta) =& \delta_{\alpha\beta}(-\lambda_j), \qquad\forall Z^\alpha,Z^\beta\in\kg_j, \quad 1\leq j\leq t, 
\end{align*}
which implies that 
$$
\kil_{\ggo_i}(Z^\alpha_i,Z^\beta_i) = \frac{1}{c_i\lambda_j}\kil_{\kg_j}(Z^\alpha,Z^\beta) = \delta_{\alpha\beta}\frac{-1}{c_i}, \qquad\forall Z^\alpha,Z^\beta\in\kg_j, \quad 1\leq j\leq t.  
$$
Thus for any $Q=y_1\kil_{\ggo_1}+\dots+y_s\kil_{\ggo_s}$, $y_i\in\RR$, we obtain that for all $i=1,\dots,s$,
\begin{equation}\label{bi}
Q(Z^\alpha_i,Z^\beta_i) = -\delta_{\alpha\beta}\frac{y_i}{c_i}, \quad\forall Z^\alpha,Z^\beta\in\kg, 
\qquad
Q(Z_i,W_i) = -\frac{y_i}{c_i}\la Z,W\ra, \quad\forall Z,W\in\kg, 
\end{equation}
and hence
\begin{equation}\label{bi2}
Q(Z,W) = -(\tfrac{y_1}{c_1}+\dots+\tfrac{y_s}{c_s})\la Z,W\ra, \qquad\forall Z,W\in\kg, 
\end{equation}
where $\ip$ is the inner product on $\kg$ such that $\{ Z^\alpha\}$ is $\ip$-orthonormal.  This proves part (iii).    We are using here that $Q(Z_i,W_i) = 0$ for all $Z\in\kg_j$, $W\in\kg_m$, where $j\ne m$, i.e., $\pi(\kg)=\pi_i(\kg_0)\oplus\dots\oplus\pi_i(\kg_t)$ is a $Q$-orthogonal decomposition, which follows from the fact that $Q|_{\pi_i(\kg)\times\pi_i(\kg)}$ is bi-invariant.  

Conversely, assume that there exists an inner product $\ip$ on $\kg$ such that $Q|_{\kg\times\kg}=-f(Q)\ip$ for any $Q=y_1\kil_{\ggo_1}+\dots+y_s\kil_{\ggo_s}$.  Thus $f$ is a linear function which is $\geq 0$ if all the $y_i$'s are $\geq 0$, so 
$$
f(y_1,\dots,y_s)=a_1y_1+\dots+a_sy_s, \qquad \mbox{for some}\quad a_1,\dots,a_s\geq 0.  
$$
On the other hand, for any nonzero $Z\in\kg$, 
$$
-f(y_1,\dots,y_s) = \tfrac{1}{\la Z,Z\ra}\sum_{i=1}^s Q(Z_i,Z_i) = \tfrac{1}{\la Z,Z\ra}\sum_{i=1}^s y_i\kil_{\ggo_i}(Z_i,Z_i), 
$$
from which follows that $a_i=-\frac{\kil_{\ggo_i}(Z_i,Z_i)}{\la Z,Z\ra}$ for all $i=1,\dots,s$.  This implies that either $\pi_i(\kg)\simeq\kg$ (if $a_i>0$, giving rise to $G'$) or $\pi_i(\kg)=0$ (if $a_i=0$, giving rise to $H$).  Moreover, $a_i=-\frac{\kil_{\kg_j}(Z,Z)}{\la Z,Z\ra}\frac{1}{c_{ij}}$ if positive for all $Z\in\kg_j$ and any $j=1,\dots,t$, that is, $c_{ij}=\lambda_j/a_i$ for some $\lambda_j>0$.  Thus $G'/K$ is aligned, concluding the proof.  
\end{proof}

In other words, aligned homogeneous spaces are those with the richest third cohomology among the class of all compact homogeneous spaces with infinite isotropy.  It follows from \eqref{cij} and \eqref{Zk0} that if $G/K$ is aligned, then a bi-invariant  $Q=y_1\kil_{\ggo_1}+\dots+y_s\kil_{\ggo_s}$ satisfies that $Q|_{\kg\times\kg}=0$ if and only if 
\begin{equation}\label{alignedQ}
\frac{1}{c_1}y_1+\dots+\frac{1}{c_s}y_s=0, 
\end{equation} 
in which case it provides a class $[H_Q]\in H^3(G/K)$.  

\begin{example}\label{ex3}
Consider $M=\SU(n_1)\times\dots\times\SU(n_s)/\SU(k_1)\times\dots\times\SU(k_t)$, where $k_1+\dots+k_t\leq n_i$ for all $i$ and the standard block diagonal embedding is always taken.  It follows from \cite[pp.37]{DtrZll} that 
$$
(c_{1j},\dots,c_{sj}) = k_j\left(\frac{1}{n_1},\dots,\frac{1}{n_s}\right),  \qquad\forall j=1,\dots,t.
$$
Thus $G/K$ is aligned and so $b_3(G/K)=s-1$ by Corollary \ref{H3-cor}.  Note that the lower bound $s-t$ provided by Corollary \ref{H3-cor}, (i) can be much smaller.  
\end{example}

\begin{example}\label{ex4}
For $M=\SO(14)\times E_6/\SU(6)\times\SO(8)$ with the standard embeddings, it follows from \cite[pp.38]{DtrZll} that 
$$
(c_{1j},c_{2j}) = \left(\unm,\unm\right),  \qquad\forall j=1,2.
$$
Thus $G/K$ is aligned with $(c_1,c_2)=(2,2)$ and $\lambda_1=\lambda_2=\unc$, which gives $b_3(G/K)=1$ by Corollary \ref{H3-cor}.  
\end{example}

\section{On reductive decompositions}

The literature on the geometry of compact homogeneous spaces $G/K$ with $G$ simple is much larger than for $G$ non-simple.  A possible explanation for this is that the isotropy representation of $G/K$ is almost never multiplicity-free in the case when $G$ is not simple.  We study such behavior in this section, which may be considered of independent interest, by describing some natural reductive decompositions for a given homogeneous space $G/K$ with $G$ non-simple.    

Given a homogeneous space $M=G/K$ as in \S\ref{H3-sec} with decompositions of $\ggo$ and $\kg$ as in 
\eqref{decs}, we consider the $\kil_{\ggo_i}$-orthogonal reductive decomposition for the homogeneous space $G_i/\pi_i(K)$, 
\begin{equation}\label{red6}
\ggo_i=\pi_i(\kg)\oplus\qg_i, \qquad\mbox{for each}\quad i=1,\dots,s.
\end{equation}
We assume from now on that the following irreducibility condition: $\pi_i(\kg)\ne 0$ for all $i=1,\dots,s$ (cf.\ Proposition \ref{al-car}, (i)).  

Any $\Ad(K)$-invariant subspace $\widetilde{\pg}$ such that
$$
\widetilde{\kg} := \pi_1(\kg)\oplus\dots\oplus\pi_s(\kg) = \kg \oplus \widetilde{\pg}
$$
therefore determines a reductive decomposition $\ggo=\kg\oplus\pg$ for $M=G/K$ given by 
\begin{equation}\label{red1}
\pg:=\pg_1\oplus\dots\oplus\pg_s\oplus\widetilde{\pg}, \qquad \mbox{where} \quad \pg_i:=(0,\dots,0,\qg_i,0,\dots,0).  
\end{equation}

For any fixed vector $\eta=(a_1,\dots,a_s)\in\RR^s$, we consider the subspace of the Lie subalgebra $\widetilde{\kg}$ given by 
\begin{equation}\label{red7}
\ggo_\eta:=\left\{ (a_1Z_1,\dots,a_sZ_s) : Z\in\kg \right\}, 
\end{equation}
Note that $\kg=\ggo_{\eta_s}$, where $\eta_s:=(1,\dots,1)$.     

It is easy to show that in the case when $\pi_i:\kg\rightarrow\ggo_i$ is in addition injective for any $i=1,\dots,s$, the subspaces $\ggo_{\eta_1},\dots,\ggo_{\eta_k}$ are independent as soon as the subset $\{\eta_1,\dots,\eta_k\}\subset\RR^s$ is linearly independent.  Moreover, for any nonzero $\eta$,  $\dim{\ggo_\eta}=\dim{\kg}$ and $\ggo_\eta$ is $\Ad(K)$-invariant and equivalent to the adjoint representation $\kg$.  All this implies that for any basis of $\RR^s$ of the form $\{\eta_1,\dots,\eta_{s-1},\eta_s\}$, the $\Ad(K)$-invariant subspace
\begin{equation}\label{red}
\widetilde{\pg} := \ggo_{\eta_1}\oplus\dots\oplus\ggo_{\eta_{s-1}}
\end{equation}
satisfies that $ \widetilde{\kg} = \kg \oplus \widetilde{\pg}$, giving rise to a reductive decomposition $\ggo=\kg\oplus\pg$ of $M=G/K$ as in \eqref{red1}.  We will use 
$$
\eta_1:=(1,-1,0,\dots,0), \quad \eta_2:=(1,1,-2,0,\dots,0), \quad \dots \quad \eta_{s-1}:=(1,\dots,1,-(s-1)),
$$  
in the rest of this section.  

We now study the case when the projections are not necessarily all injective.  It is easy to check that 
$$
\widetilde{\kg}=\widetilde{\kg_0}\oplus\dots\oplus\widetilde{\kg_t}, \qquad\mbox{where}\quad 
\widetilde{\kg_j}:=\pi_1(\kg_j)\oplus\dots\oplus\pi_s(\kg_j), \qquad\forall j=0,\dots,t.    
$$
For each $j=1,\dots,t$, consider the nonempty subset 
$$
\{ i_1,\dots,i_{n_j}\} := \{ 1\leq i\leq s: \pi_i(\kg_j)\ne 0\}.  
$$
Any $\eta=(a_1,\dots,a_{n_j})\in\RR^{n_j}$ defines the following $\Ad(K)$-invariant subspace of $\widetilde{\kg_j}$,
$$
\ggo^j_{\eta}:=\left\{ a_1Z_{i_1}+\dots+a_{n_j}Z_{i_{n_j}} : Z\in\kg_j \right\}, 
$$
which is equivalent to $\kg_j$ as an $\Ad(K)$-representation.  Thus for the basis of $\RR^{n_j}$ given by $\{(1,\dots,1),\eta_1,\dots,\eta_{n_j-1}\}$, we have that
\begin{equation}\label{red5}
\widetilde{\kg_j} := \kg_j\oplus\ggo^j_{\eta_1}\oplus\dots\oplus\ggo^j_{\eta_{n_j-1}}.
\end{equation}
On the other hand, if $\{ Z^1,\dots,Z^{d_0}\}$ is any basis of $\kg_0$, then we proceed as above for each of the subspaces $\RR Z^k$ (instead of $\kg_j$) to obtain a decomposition $\widetilde{\kg_0} = \kg_0\oplus\pg_0$, where
\begin{equation}\label{red4}
\pg_0 := \hg^1_{\eta_1}\oplus\dots\oplus\hg^1_{\eta_{m_1-1}} \oplus \dots\oplus \hg^{d_0}_{\eta_1}\oplus\dots\oplus\hg^{d_0}_{\eta_{m_{d_0}-1}},
\end{equation}
$\hg^\alpha_{\eta}:= \RR\left(a_1Z^\alpha_{i_1}+\dots+a_{m_k}Z^\alpha_{i_{m_\alpha}}\right)$ for any $\eta=(a_1,\dots,a_{m_\alpha})\in\RR^{m_\alpha}$,
$$
\{ i_1,\dots,i_{m_\alpha}\} := \{ 1\leq i\leq s: \pi_i(Z^\alpha)\ne 0\}, \qquad \forall \alpha=1,\dots,d_0.
$$  
This provides a reductive decomposition $\ggo=\kg\oplus\pg$ of the homogeneous space $M=G/K$, where 
\begin{equation}\label{red-gen}
\pg = \pg_1\oplus\dots\oplus\pg_s \oplus \pg_0 \oplus  \ggo^1_{\eta_1}\oplus\dots\oplus\ggo^1_{\eta_{n_1-1}} \oplus \dots\oplus \ggo^t_{\eta_1}\oplus\dots\oplus\ggo^t_{\eta_{n_t-1}},
\end{equation} 
where $\pg_i$ is as in \eqref{red1} and $\pg_0$ is the $(m_1+\dots+m_{d_0}-d_0)$-dimensional subspace given in \eqref{red4}.   As $\Ad(K)$-representations, $\pg_i$ is the isotropy representation of the homogeneous space $G_i/\pi_i(K)$, each $\hg^\alpha_{\eta_i}\subset\pg_0$ is a one-dimensional trivial representation and any $\ggo^j_{\eta_i}$ is equivalent to the irreducible representation $\kg_j$.  In particular, $G/K$ is not multiplicity-free under any of the following conditions: 
\begin{enumerate}[a)]
\item $G_i/\pi_i(K)$ is not multiplicity-free.  

\item $d_0\geq 2$. 

\item $d_0=1$ and $m_\alpha\geq 3$ for some $\alpha=1,\dots,d_0$. 

\item $n_j\geq 3$ for some $j=1,\dots,t$.  
\end{enumerate}

Note that if all the projections are injective, then $n_1=\dots=n_t=s$, $m_1=\dots=m_{d_0}=s$ and each summand $\ggo_{\eta_i}$, $i=1,\dots,s-1$ in decomposition \eqref{red} can be written as
\begin{equation}\label{red3}
\ggo_{\eta_i} = \hg^1_{\eta_i}\oplus\dots\oplus\hg^{d_0}_{\eta_i}\oplus\ggo^1_{\eta_i}\oplus\dots\oplus\ggo^t_{\eta_i}.    
\end{equation}

Another natural choice of a reductive decomposition for $G/K$ is the $g_b$-orthogonal reductive decomposition $\ggo=\kg\oplus\pg$, where $g_b$ is any bi-invariant metric on $G$, in which case $\widetilde{\pg}$ is necessarily the $g_b$-orthogonal complement of $\kg$ in $\pi_1(\kg)\oplus\dots\oplus\pi_s(\kg)$.  Note that additionally, $\widetilde{\pg}$ can be decomposed as $\widetilde{\pg}=\pg_{s+1}\oplus\dots\oplus\pg_{r}$ in $g_b$-orthogonal $\Ad(K)$-invariant subspaces, producing a $g_b$-orthogonal decomposition of $\pg$, 
\begin{equation}\label{red2}
\pg = \pg_1\oplus\dots\oplus\pg_s\oplus \pg_{s+1}\oplus\dots\oplus\pg_{r},   
\end{equation}
in $\Ad(K)$-invariant subspaces.

\subsection{Aligned case} 
The above two choices of reductive decompositions coincide in the aligned case.  

\begin{proposition}\label{red-al}
Assume that $M=G/K$ is aligned with positive constants $(c_1,\dots,c_s)$ and consider a bi-invariant metric $g_b=z_1(-\kil_{\ggo_1})+\dots+z_s(-\kil_{\ggo_s})$, $z_1,\dots,z_s>0$ and the corresponding $g_b$-orthogonal reductive decomposition $\ggo=\kg\oplus\pg$.  Then the following is a $g_b$-orthogonal decomposition of $\pg$ in $\Ad(K)$-invariant subspaces: 
$$
\pg = \pg_1\oplus\dots\oplus\pg_s\oplus \pg_{s+1}\oplus\dots\oplus\pg_{2s-1},  
$$
where $\pg_i=(0,\dots,0,\qg_i,0,\dots,0)$ for all $i=1,\dots,s$ (see \eqref{red6}), 
$$
\pg_{s+1}=\ggo_{\eta_1}, \quad \dots\quad \pg_{2s-1}=\ggo_{\eta_{s-1}} \qquad\mbox{(see \eqref{red7})}
$$ 
and
$$
\eta_j:=\left(1,\dots,1,-\tfrac{c_{j+1}}{z_{j+1}}\left(\tfrac{z_1}{c_1}+\dots+\tfrac{z_j}{c_j}\right),0,\dots,0\right) \qquad 
(1\; \mbox{appears $j$-times}),  
$$
for all $j=1,\dots,s-1$.  In turn, each $\ggo_{\eta_j}$ $g_b$-orthogonally decomposes in $\Ad(K)$-irreducible invariant subspaces as 
$$
\ggo_{\eta_j} = \ggo^0_{\eta_j} \oplus\ggo^1_{\eta_j}\oplus\dots\oplus\ggo^t_{\eta_j}, 
$$ 
where $\ggo^l_\eta:=\left\{ (a_1Z_1,\dots,a_sZ_s) : Z\in\kg_l \right\}$ for any $\eta=(a_1,\dots,a_s)\in\RR^s$.   
\end{proposition}

\begin{remark}
As $\Ad(K)$-representations, $\pg_i\simeq\qg_i$, the isotropy representation of the homogeneous space $G_i/\pi_i(K)$, $\ggo_{\eta_j}\simeq\kg$ and $\ggo^j_{\eta_i}\simeq\kg_j$.  In particular, an aligned $G/K$ is never multiplicity-free if $s\geq 3$.  
\end{remark}

\begin{remark}
If $g_b$ is the standard metric $\gk$ (i.e., $z_1=\dots=z_s=1$), then each $\eta_j$ defines a closed $3$-form $H_{Q_j}$, where $Q_j$ is the linear combination of the $\kil_{\ggo_i}$'s using the coordinates of $\eta_j$ as the coeficcients $y_i$'s.  In particular, $\{ [H_{Q_1}],\dots,[H_{Q_{s-1}}]\}$ is a basis for $H^3(G/K)$ (see Theorem \ref{H3}, (iii)).   
\end{remark}

\begin{proof}
If $i<j$ and $A_{s+j}:= -\frac{c_{j+1}}{z_{j+1}}\left(\frac{z_1}{c_1}+\dots+\frac{z_j}{c_j}\right)$, then for any $Z,W\in\kg_m$, $1\leq m\leq s$,
\begin{align*}
& g_b\left((Z_1,\dots,Z_i,A_{s+i}Z_{i+1},0,\dots,0), (W_1,\dots,W_j,A_{s+j}W_{j+1},0,\dots,0)\right) \\ 
=& \sum_{l=1}^i g_b(Z_l,W_l) + A_{s+i}g_b(Z_{i+1},W_{i+1}) 
=-\sum_{l=1}^i z_l\kil_{\ggo_l}(Z_l,W_l) - A_{s+i}z_{i+1}\kil_{\ggo_{i+1}}(Z_{i+1},W_{i+1}) \\ 
=&-\kil_\kg(Z,W)\tfrac{1}{\lambda_m}\left(\sum_{l=1}^i \tfrac{z_l}{c_l} + A_{s+i}\tfrac{z_{i+1}}{c_{i+1}}\right) = 0. 
\end{align*}
The same holds for any $Z,W\in\kg_0$, which follows in much the same way by using Definition \ref{aligned}, (ii).  Finally, for $Z\in\kg_m$ and $W\in\kg_n$, $m\ne n$, we use that $\pi_i(\kg)=\pi_i(\kg_0)\oplus\dots\oplus\pi_i(\kg_t)$ is a $g_b$ orthogonal decomposition since $g_b|_{\pi_i(\kg)\times\pi_i(\kg)}$ is bi-invariant.  We conclude that $g_b(\ggo_{\eta_i},\ggo_{\eta_j})=0$ for all $i\ne j$, as was to be shown.  
\end{proof}

\section{Harmonic $3$-forms on compact homogeneous spaces}\label{harm-sec} 

Let $M=G/K$ be a homogeneous space, where $G$ is compact, connected and semisimple and $K$ is connected.  We have seen in \S\ref{H3-sec} that, given a reductive decomposition, there is a canonical closed $3$-form $H_Q$ attached to each bi-invariant symmetric bilinear form $Q$ on $\ggo$ such that $Q|_{\kg\times\kg}=0$.  For any given bi-invariant metric $g_b$ on $G$,  let us fix the $g_b$-orthogonal reductive decomposition $\ggo=\kg\oplus\pg$ and as a background $G$-invariant metric the normal metric $g_b|_{\pg\times\pg}$.   

The well-known fact that any bi-invariant $3$-form on a compact Lie group is harmonic with respect to any bi-invariant metric motivates the following natural problem:
\begin{quote}
Is any closed $3$-form $H_Q$ $g_b$-harmonic?  If not, what are the conditions in terms of algebraic invariants of the groups $G$, $K$ and the embedding $K\subset G$ to be fulfilled by $Q$ and a normal metric $g_b$ on $M=G/K$ in order for
$H_Q$ to be $g_b$-harmonic?  Can some other $G$-invariant metric $g$ (not necessarily normal) satisfy that $H_Q$ is $g$-harmonic for every $Q$?  
\end{quote}

We first note that for any bi-invariant symmetric bilinear form $Q$ on $\ggo$, 
\begin{equation}\label{Qbcc} 
\mbox{the $3$-form $\widetilde{Q}=Q([\cdot,\cdot],\cdot)\in(\Lambda^3\pg^*)^K$ is $g_b$-coclosed}.
\end{equation}
Indeed, for any $\beta\in(\Lambda^2\pg^*)^K$,  
$$
g_b(\widetilde{Q},d\beta) = g_b(\hat{\widetilde{Q}},\hat{d}\hat{\beta}) =  g_b(\overline{Q},\hat{d}\hat{\beta}) = g_b(\hat{d}_{g_b}^*\overline{Q},\hat{\beta}) =0,
$$
since $g_b$ and $\overline{Q}$ are both bi-invariant on $G$ and so $\hat{d}_{g_b}^*\overline{Q}=0$.     

For any $g_b$-orthogonal decomposition,
\begin{equation}\label{dec2}
\pg=\pg_1\oplus\dots\oplus\pg_r,
\end{equation}
in $\Ad(K)$-invariant subspaces (not necessarily irreducible), we take a $g_b$-orthonormal basis $\{ e_\alpha^i\}$ of $\pg_i$ for each $i=1,\dots,r$ and consider the corresponding structural constants 
$$
c_{i\alpha,j\beta}^{k\gamma}:= g_b([e_\alpha^i,e_\beta^j],e_\gamma^k).  
$$ 
Note that $c_{i\alpha,j\beta}^{k\gamma}=-c_{j\beta,i\alpha}^{k\gamma}=-c_{i\alpha,k\gamma}^{j\beta}$.  
Given any $G$-invariant metric $g$, there exists at least one $g_b$-orthogonal decomposition as in \eqref{dec2}, which is also $g$-orthogonal, such that $g$ has the form
\begin{equation}\label{metric}
g=x_1g_b|_{\pg_1\times\pg_1}+\dots+x_rg_b|_{\pg_r\times\pg_r}, \qquad x_i>0.
\end{equation}
This metric will be denoted by 
$$
g=(x_1,\dots,x_r)_{g_b}.
$$ 
Note that the background normal metric is $g_b=(1,\dots,1)_{g_b}$ and $\{ \frac{1}{\sqrt{x_i}}e^i_\alpha\}$ is a $g$-orthonormal basis of $\pg_i$.  

We now consider a $g_b$-orthogonal reductive decomposition as in \eqref{red2} (recall also the decompositions \eqref{decs} and \eqref{red1}), so 
$$
\pg=\overline{\pg}\oplus\widetilde{\pg}, \qquad \mbox{where}\quad \overline{\pg}:=\pg_1\oplus\dots\oplus\pg_s, \quad \widetilde{\pg}:=\pg_{s+1}\oplus\dots\oplus\pg_r.  
$$
It is easy to check that for all $1\leq i\ne j\leq s$ and $s+1\leq k,l\leq r$,  
\begin{equation}\label{pipj}
[\pg_i,\pg_i]\subset\kg+\pg_i+\widetilde{\pg}, \qquad [\pg_i,\pg_j]=0, \qquad  [\pg_{k},\pg_i]\subset\pg_i, \qquad [\pg_{k},\pg_{l}]\subset \kg+\widetilde{\pg}.
\end{equation}
We have that the bi-invariant metric is given by 
$$
g_b=z_1(-\kil_{\ggo_1})+\dots+z_s(-\kil_{\ggo_s}), \qquad \mbox{for some}\quad z_1,\dots,z_s>0, 
$$
and recall that any bi-invariant symmetric bilinear form $Q$ such that $Q|_{\kg\times\kg}=0$, say, 
$$
Q=y_1\kil_{\ggo_1}+\dots+y_s\kil_{\ggo_s}, \qquad y_1,\dots,y_s\in\RR,
$$ 
defines a closed $3$-form $H_Q$ as in \eqref{HQ}.  Since
\begin{equation}\label{pipj2}
Q(\pg_i,\pg_j)=Q(\pg_i,\pg_{k})=Q(\pg_i,\kg)=0, \qquad\forall i\ne j\leq s< k,   
\end{equation}
the $3$-form $H_Q$ is nonzero (up to permutations) only on 
$$
\pg_i\times\pg_i\times\pg_i, \qquad \pg_i\times\pg_i\times\pg_j, \qquad \pg_j\times\pg_k\times\pg_l, \qquad i\leq s < j,k,l,
$$ 
and since by \eqref{pipj2} the subscript $\pg$ can be deleted in formula \eqref{HQ} when the three vectors are in $\pg_i$, $i\leq s$, we obtain that
\begin{equation}\label{HQ1}
H_Q(X,Y,Z) = Q([X,Y],Z) = -\frac{y_i}{z_i}g_b([X,Y],Z), 
\qquad\forall X,Y,Z\in\pg_i, \quad i\leq s.    
\end{equation}
We now show that the subspace $(\Lambda^2\overline{\pg}^*)^K$ imposes no condition on the $g$-coclosedness of $H_Q$.  

\begin{proposition}\label{ompt1}
$g(H_Q,d\omega)=0$ for any $\omega\in(\Lambda^2\pg^*)^K$ such that $\omega(\widetilde{\pg},\cdot)=0$ and any metric $g=(x_1,\dots,x_r)_{g_b}$. 
\end{proposition}

\begin{proof}
If $H\in(\Lambda^3\pg^*)^K$ and $\omega\in(\Lambda^2\pg^*)^K$, then  
\begin{align}
g(H,d\omega) =& \sum_{\substack{\alpha,\beta,\gamma\\ i,j,k}} \tfrac{1}{x_ix_jx_k}H(e_\alpha^i,e_\beta^j,e_\gamma^k)d\omega(e_\alpha^i,e_\beta^j,e_\gamma^k) \notag\\ 
=&  \sum_{\substack{\alpha,\beta,\gamma\\ i,j,k}} \tfrac{1}{x_ix_jx_k}H(e_\alpha^i,e_\beta^j,e_\gamma^k) \left(-\omega([e_\alpha^i,e_\beta^j]_\pg,e_\gamma^k) +\omega([e_\alpha^i,e_\gamma^k]_\pg,e_\beta^j) -\omega([e_\beta^j,e_\gamma^k]_\pg,e_\alpha^i)\right) \notag\\ 
=&  -3\sum_{\substack{\alpha,\beta,\gamma\\ i,j,k}} \tfrac{1}{x_ix_jx_k}H(e_\alpha^i,e_\beta^j,e_\gamma^k)\omega([e_\alpha^i,e_\beta^j]_\pg,e_\gamma^k).  \label{Hdo1}
\end{align}
It therefore follows from \eqref{pipj} that for $\omega\in(\Lambda^2\pg^*)^K$ such that $\omega(\widetilde{\pg},\cdot)=0$, 
\begin{align*}
g(H_Q,d\omega) 
=& -3\sum_{\substack{\alpha,\beta,\gamma\\ i\leq s}} \tfrac{1}{x_i^3}H_Q(e_\alpha^i,e_\beta^i,e_\gamma^i)\omega([e_\alpha^i,e_\beta^i]_\pg,e_\gamma^i) \\ 
& -6\sum_{\substack{\alpha,\beta,\gamma\\ i\leq s<j}} \tfrac{1}{x_i^2x_j}H_Q(e_\alpha^j,e_\beta^i,e_\gamma^i)\omega([e_\alpha^j,e_\beta^i]_\pg,e_\gamma^i).
\end{align*}
The second summand always vanishes since if $\pi_i(e_\alpha^j)=\pi_i(Z)$, $Z\in\kg$, then
\begin{align*}
\omega([e_\alpha^j,e_\beta^i]_\pg,e_\gamma^i) =& \omega([\pi_i(Z),e_\beta^i],e_\gamma^i) = \omega([Z,e_\beta^i],e_\gamma^i) = \omega([Z,e_\gamma^i],e_\beta^i) \\ 
=& \omega([\pi_i(Z),e_\gamma^i],e_\beta^i) = \omega([e_\alpha^j,e_\gamma^i]_\pg,e_\beta^i), \qquad\forall i\leq s,   
\end{align*}
that is, $\omega([e_\alpha^j,e_\beta^i]_\pg,e_\gamma^i)$ is symmetric as a function of $(\beta,\gamma)$.  On the other hand, according to \eqref{HQ1}, the first summand equals  
$$
3\sum_{\substack{\alpha,\beta,\gamma\\ i\leq s}} \frac{y_i}{z_ix_i^3}c_{i\alpha,i\beta}^{i\gamma}\omega([e_\alpha^i,e_\beta^i]_\pg,e_\gamma^i) = 3\sum_{\substack{\alpha,\beta,\gamma,\delta\\ i\leq s}} \frac{y_i}{z_ix_i^3}c_{i\alpha,i\beta}^{i\gamma} c_{i\alpha,i\beta}^{i\delta}\omega(e_\delta^i,e_\gamma^i),  
$$
which vanishes since $c_{k\alpha,k\beta}^{k\gamma}c_{k\alpha,k\beta}^{k\delta}$ and $\omega(e_\delta^k,e_\gamma^k)$ are respectively symmetric and skew-symmetric as functions of $(\delta,\gamma)$.
\end{proof} 

For each $i=1,\dots,s$, we consider a decomposition of the isotropy representation of $G_i/\pi_i(K)$, 
$$
\qg_i=\qg_i^0\oplus\qg_i^1\oplus\dots\oplus\qg_i^{r_i},
$$
in $\Ad(K)$-invariant subspaces, where $\qg_i^0$ is the trivial representation (i.e., $[\kg,\qg_i^0]=0$) and $\qg_i^k$ is irreducible and non-trivial for all $k=1,\dots,r_i$.   

If we assume that the following two conditions hold:
\begin{enumerate}[{\rm (i)}]
\item $\qg_i^k$ and $\kg_j$ are not equivalent as $K$-representations for all $i=1,\dots,s$, $k=1,\dots,r_i$ and $j=1,\dots,t$;  

\item either $\qg_i^0=0$ for all $i=1,\dots,s$ or $\kg_0=0$,  
\end{enumerate} 
then the above proposition implies that the conditions on $Q$ and $g=(x_1,\dots,x_r)_{g_b}$ in order for $H_Q$ to be $g$-coclosed will only come out from the linear system of equations given by 
\begin{equation}\label{h}
g(H_Q,d\omega)=0, \qquad\forall \omega\in(\Lambda^2\pg^*)^K, \quad \omega(\overline{\pg},\cdot)=0.
\end{equation}  
Indeed, conditions (i) and (ii) above imply that $\omega(\overline{\pg},\widetilde{\pg})=0$ for any $\omega\in(\Lambda^2\pg^*)^K$ since any two different isotypic components are necessarily $\omega$-orthogonal.  

The following are examples of spaces which can not appear as a $G_i/\pi_i(K)$ in order for $G/K$ to satisfy condition (i) above (they were generously provided to the authors by C. B\"ohm and W. Ziller, respectively).  

\begin{example}
Assume that $G_i$ contains a subgroup of the form $H\times H$ and consider $G_i/K$, where $K:=\Delta H\subset H\times H\subset G$.  Thus the adjoint representation $\kg\simeq\hg$ also appears in $\qg_i$.  
\end{example}

\begin{example}
If $G_i/\pi_i(K)=\SO(n)/\SO(3)$, then the $3$-dimensional irreducible factors in the standard reductive complement $\qg_i$ are necessarily equivalent to the adjoint representation $\kg=\sog(3)$ .  
\end{example}

\subsection{Case $s=2$}\label{s2-sec}
In the case when $\ggo$ has only two simple factors, we have that $\pg=\pg_1\oplus\pg_2\oplus\pg_3$, where, as an $\Ad(K)$-representation, $\pg_3$ is equivalent to an ideal $\hg$ of $\kg$, say via an equivalence map $\vp:\hg\rightarrow\pg_3$.  This implies that any $\omega\in(\Lambda^2\pg_3^*)^K$ vanishes on $\vp([\hg,\hg])\cap\pg_3$ and so $d\omega=0$.  It follows from Proposition \ref{ompt1} that $H_Q$ is $g$-harmonic for any $g=(x_1,x_2,x_3)_{g_b}$ and any $g_b$, provided that conditions (i) and (ii) above hold.  Note that if in addition $\qg_1$ and $\qg_2$ are $\Ad(K)$-irreducible and $\kg$ is either simple or one-dimensional, then any $G$-invariant metric is of the form $g=(x_1,x_2,x_3)_{g_b}$.

\section{Aligned case}\label{harmal-sec}

In this section, we continue the study of the problem stated at the beginning of \S\ref{harm-sec}, under the assumption that $M=G/K$ is aligned (see Definition \ref{aligned}).  We therefore need to focus on condition \eqref{h}.  

We consider the $g_b$-orthogonal reductive decomposition $\ggo=\kg\oplus\pg$ provided by Proposition \ref{red-al}, which is given by   
$$
\pg =\overline{\pg}\oplus \widetilde{\pg}, \qquad \mbox{where}\quad \overline{\pg}:=\pg_1\oplus\dots\oplus\pg_s, \quad \widetilde{\pg}:=\ggo_{\eta_1}\oplus\dots\oplus\ggo_{\eta_{s-1}}.
$$
Recall that all these subspaces are pairwise $g_b$-orthogonal.  It is easy to see that in addition to the Lie bracket properties given in \eqref{pipj}, we have that 
\begin{equation}\label{getai}
\begin{array}{l}
[\pg_i,\pg_i]\subset \pg_i+\ggo_{\eta_{i-1}}+\dots+\ggo_{\eta_{s-1}}+\kg, \qquad\forall i\leq s, 
 \\ \\  
 
[\ggo_{\eta_i},\ggo_{\eta_i}] \subset\ggo_{\eta_i}+\dots+\ggo_{\eta_{s-1}}+\kg, \qquad\forall i\leq s, 
 \qquad [\ggo_{\eta_i},\ggo_{\eta_j}]\subset\ggo_{\eta_i}, \qquad\forall i<j\leq s.
\end{array}
\end{equation}
Recall from \eqref{alignedQ} that, in order to define a closed $3$-form $H_Q$, a bi-invariant symmetric bilinear form $Q=y_1\kil_{\ggo_1}+\dots+y_s\kil_{\ggo_s}$ must satisfy 
$$
\tfrac{y_1}{c_1}+\dots+\tfrac{y_s}{c_s} = 0.
$$ 
The following notation will be strongly used from now on without any further reference: for $j=1,\dots,s-1$, 
$$
\begin{array}{lcl}
A_{s+j}:=-\frac{c_{j+1}}{z_{j+1}}\left(\frac{z_1}{c_1}+\dots+\frac{z_j}{c_j}\right), &\qquad& B_{s+j}:=\frac{z_1}{c_1}+\dots+\frac{z_j}{c_j}+A_{s+j}^2\frac{z_{j+1}}{c_{j+1}}, \\  \\ 
B_{2s}:=\frac{z_1}{c_1}+\dots+\frac{z_s}{c_s}, &\qquad& C_{s+j}:=\frac{y_1}{c_1}+\dots+\frac{y_j}{c_j}+A_{s+j}\frac{y_{j+1}}{c_{j+1}},  \\ \\
D_{s+j}:=\frac{z_1}{c_1}+\dots+\frac{z_j}{c_j}+A_{s+j}^3\frac{z_{j+1}}{c_{j+1}},&\qquad& 
D_{2s}:=B_{2s}, \quad C_{2s}:=0.
\end{array}
$$ 
We consider the inner product $\ip=-\kil_{\ggo}|_{\kg\times\kg}$ on $\kg$ (cf.\ Proposition \ref{al-car}, (iii)) together with a $\ip$-orthonormal adapted basis $\{ Z^1,\dots,Z^{\dim{\kg}}\}$ of $\kg=\kg_0\oplus\kg_1\oplus\dots\oplus\kg_t$.  It follows from \eqref{bi} that 
\begin{equation}\label{gbi}
g_b(Z^\alpha_i,Z^\beta_i) = \delta_{\alpha\beta}\frac{z_i}{c_i}, \qquad\forall Z^\alpha,Z^\beta\in\kg, 
\qquad
g_b(Z_i,W_i) = \frac{z_i}{c_i}\la Z,W\ra, \qquad\forall Z,W\in\kg, 
\end{equation}
and from \eqref{bi2} that
\begin{equation}\label{gbi2}
g_b(Z,W) = B_{2s}\la Z,W\ra, \qquad\forall Z,W\in\kg.  
\end{equation}
If $\vp_i:\kg\rightarrow\ggo_{\eta_i}$ is the $\Ad(K)$-equivariant isomorphism given by 
$$
\vp_i(Z):=\left(Z_1,\dots,Z_i,A_{s+i}Z_{i+1},0,\dots,0\right), \qquad\forall Z\in\kg, \quad i=1,\dots,s,  
$$
(note that $\vp_s:\kg\rightarrow\ggo_{\eta_s}=\kg$ is the identity map) and we set 
$$
e_\alpha^{s+i}:=\tfrac{1}{\sqrt{B_{s+i}}}\vp_i(Z^\alpha), 
$$ 
then it follows from \eqref{gbi} that $\{e_\alpha^{s+i}:\alpha=1,\dots,\dim{\kg}\}$ is a $g_b$-orthonormal basis of 
$\ggo_{\eta_i}$ for $1\leq i \leq s$ (note that $e_\alpha^{2s}=\tfrac{1}{\sqrt{B_{2s}}}Z^\alpha$) and so
\begin{equation}\label{basis}
\{e_\alpha^{s+i}:i=1,\dots,s, \; \alpha=1,\dots,\dim{\kg}\}
\end{equation}
is a $g_b$-orthonormal adapted basis of the subalgebra $\widetilde{\kg}=\widetilde{\pg}\oplus\kg=\ggo_{\eta_1}\oplus\dots\oplus\ggo_{\eta_{s-1}}\oplus\ggo_{\eta_{s}}$.  We also take a $g_b$-orthonormal adapted basis 
\begin{equation}\label{basis2}
\{e_\alpha^i:i=1,\dots,s, \; \alpha=1,\dots,\dim{\pg_i}\}
\end{equation} 
of $\overline{\pg}$.   

\begin{lemma}\label{ompt2}
\quad  
\begin{enumerate}[{\rm (i)}] 
\item The nonzero structural constants of the basis $\{ e^i_\alpha\}$ of $\ggo$ given in \eqref{basis} and \eqref{basis2} other than $c_{i\alpha,i\beta}^{i\gamma}$, $i\leq s$, are given by  
$$
c_{i\alpha,i\beta}^{j\gamma} =\left\{\begin{array}{ll} 
\tfrac{\sqrt{B_{2s}}}{\sqrt{B_j}}c_{i\alpha,i\beta}^{2s\gamma}, & i\leq s, \; i+s\leq j< 2s, \\ \\ 
\tfrac{A_{i+s-1}\sqrt{B_{2s}}}{\sqrt{B_j}}c_{i\alpha,i\beta}^{2s\gamma}, & i\leq s,\; j=i+s-1,  \\ \\ 
0, & i\leq s<j< i+s-1,  
\end{array}\right. 
c_{i\alpha,i\beta}^{j\gamma} =\left\{\begin{array}{ll} 
\frac{1}{\sqrt{B_j}}\overline{c}_{\alpha\beta}^\gamma, & s<i<j\leq 2s, \\ \\ 
\frac{D_i}{\sqrt{B_i^3}} \overline{c}_{\alpha\beta}^\gamma, & s<i=j\leq 2s,  
\end{array}\right.
$$
where $\overline{c}_{\alpha\beta}^\gamma$ are the structural constants of $\kg$ with respect to the basis $\{ Z^1,\dots,Z^{\dim{\kg}}\}$.  

\item $Q(e_\alpha^{2s},e_\beta^j) = \delta_{\alpha\beta} \tfrac{-C_j}{\sqrt{B_jB_{2s}}}$, for all $j=s+1,\dots,2s-1$.  
\end{enumerate}
\end{lemma}

\begin{proof}
For $i\leq s<j< 2s$ we have that,  
\begin{align*}
c_{i\alpha,i\beta}^{j\gamma} 
=& \tfrac{1}{\sqrt{B_j}} g_b((0,\dots,0,[e_\alpha^i,e_\beta^i],0,\dots,0), (Z^\gamma_1,\dots,Z^\gamma_{j-s},A_jZ^\gamma_{j-s+1},0\dots,0)),
\end{align*}
and thus the formula in part (i) follows from the fact that 
$$
g_b([e_\alpha^i,e_\beta^i],Z^\gamma_i) = g_b([e_\alpha^i,e_\beta^i],Z^\gamma) = \sqrt{B_{2s}} g_b([e_\alpha^i,e_\beta^i],e^{2s}_\gamma) = \sqrt{B_{2s}}c_{i\alpha,i\beta}^{2s\gamma}.  
$$
Using \eqref{gbi}, for $s<i<j\leq 2s$ we obtain,  
\begin{align*}
c_{i\alpha,i\beta}^{j\gamma} =& g_b([e_\alpha^i,e_\beta^i],e_\gamma^j) 
= \tfrac{1}{B_i\sqrt{B_j}} g_b([\vp_{i-s}(Z^\alpha),\vp_{i-s}(Z^\beta)],\vp_{j-s}(Z^\gamma))  \\ 
=& \tfrac{1}{B_i\sqrt{B_j}} g_b\left(([Z^\alpha_1,Z^\beta_1],\dots,[Z^\alpha_{i-s},Z^\beta_{i-s}]),A_i^2[Z^\alpha_{i-s+1},Z^\beta_{i-s+1}],0\dots,0),\vp_{j-s}(Z^\gamma)\right)  \\
=& \tfrac{1}{B_i\sqrt{B_j}} \left(\sum_{l=1}^{i-s} g_b([Z^\alpha,Z^\beta]_l,Z^\gamma_l)) +A_i^2g_b([Z^\alpha,Z^\beta]_{i-s+1},Z^\gamma_{i-s+1})\right)  \\
=& \tfrac{1}{B_i\sqrt{B_j}} \left(\sum_{l=1}^{i-s} \frac{z_l}{c_l} +A_i^2\frac{z_{i-s+1}}{c_{i-s+1}}\right)\la[Z^\alpha,Z^\beta],Z^\gamma\ra = \tfrac{1}{\sqrt{B_j}}\overline{c}_{\alpha\beta}^\gamma. 
\end{align*}
In much the same way, if $s<i=j\leq 2s$ then
$$
c_{i\alpha,i\beta}^{i\gamma} = \tfrac{1}{B_i\sqrt{B_i}} \left(\sum_{l=1}^{i-s} \frac{z_l}{c_l} +A_i^3\frac{z_{i-s+1}}{c_{i-s+1}}\right)\la[Z^\alpha,Z^\beta],Z^\gamma\ra = \tfrac{D_i}{\sqrt{B_i^3}}\overline{c}_{\alpha\beta}^\gamma.
$$
Finally, part (ii) follows from the following computation:
\begin{align*}
Q(e_\alpha^{2s},e_\beta^j) =& \tfrac{1}{\sqrt{B_{2s}B_j}}Q(Z^\alpha,\vp_{j-s}(Z^\beta)) \\ 
=& \tfrac{1}{\sqrt{B_{2s}B_j}}\left(\sum_{l=1}^{j-s}y_l\kil_{\ggo_l}(Z^\alpha_l,Z^\beta_l) + y_{j-s+1}A_{j}\kil_{\ggo_{j-s+1}}(Z^\alpha_{j-s+1},Z^\beta_{j-s+1})\right) \\
=& \tfrac{1}{\sqrt{B_{2s}B_j}}\left(\sum_{l=1}^{j-s}y_l\delta_{\alpha\beta}\frac{-1}{c_l}+ y_{j-s+1}A_{j}\delta_{\alpha\beta}\frac{-1}{c_{j-s+1}}\right) 
=\delta_{\alpha\beta} \tfrac{-C_j}{\sqrt{B_{2s}B_j}}, 
\end{align*}
concluding the proof.  
\end{proof} 

According to Proposition \ref{red-al}, if we set 
$$
\widetilde{\pg}_0:=\ggo_{\eta_1}^0+\dots+\ggo_{\eta_{s-1}}^0, \qquad \widetilde{\pg}_{\geq 1}:=\sum_{j=1}^t \ggo_{\eta_1}^j+\dots+\ggo_{\eta_{s-1}}^j,
$$
then $\widetilde{\pg}=\widetilde{\pg}_0\oplus\widetilde{\pg}_{\geq 1}$ and it is easy to check that
\begin{equation}\label{ptpt}
[\widetilde{\pg},\widetilde{\pg}]\subset\widetilde{\pg}_{\geq 1}+\kg.
\end{equation}

\begin{lemma}\label{ompt3} \hspace{1cm}
\begin{enumerate}[{\rm (i)}]
\item Any skew-symmetric $(s-1)\times(s-1)$ matrix $[\omega_{ij}]$, $s<i,j\leq 2s-1$, defines a $2$-form $\omega\in(\Lambda^2\pg^*)^K$ such that $\omega(\overline{\pg},\cdot)=0$ and which on $\widetilde{\pg}\times\widetilde{\pg}$ is given by,
$$
\omega(\vp_{i-s}(Z),\vp_{j-s}(W)) := \omega_{ij}g_b(Z,W), \qquad\forall Z,W\in\kg, \quad s< i,j\leq 2s-1.
$$

\item The above $2$-form satisfies that  
$$
\omega(e_\alpha^i,e_\beta^j) = \delta_{\alpha\beta} \omega_{ij} \tfrac{B_{2s}}{\sqrt{B_iB_j}}, \qquad \forall s< i,j\leq 2s-1.  
$$
\item 
Any $\omega\in(\Lambda^2\pg^*)^K$ such that $\omega(\overline{\pg},\cdot)=0$ is of the form $\omega=\omega_1+\omega_2$, where $\omega_1$ is as above, $\omega_2(\widetilde{\pg}_{\geq 1},\cdot)=0$ and $d\omega_2=0$.  
\end{enumerate}  
\end{lemma}

\begin{proof}
It follows from the $\Ad(K)$-equivariance of each $\vp_{i-s}$ that for any $U,Z,W\in\kg$, 
\begin{align*}
\omega([U,\vp_{i-s}(Z)],\vp_{j-s}(W)) =& \omega(\vp_{i-s}([U,Z]),\vp_{j-s}(W)) = \omega_{ij}g_b([U,Z],W) \\ 
=& -\omega_{ij}g_b(Z,[U,W]) = -\omega(\vp_{i-s}(Z),\vp_{j-s}([U,W])) \\ 
=& -\omega(\vp_{i-s}(Z),[U,\vp_{j-s}(W)]),
\end{align*}
which shows that $\omega$ is $\Ad(K)$-invariant and so part (i) holds.  

Part (ii) follows from \eqref{gbi2} and part (iii) from the fact that $(\Lambda^2\ggo_{\eta_i}^j)^K=0$ for all $i$ and $j\geq 1$ (recall that the $K$-representation $\ggo_{\eta_i}^j$ is equivalent to the adjoint representation of the simple Lie algebra $\kg_j$) and so $(\Lambda^2(\ggo_{\eta_i}^1+\dots+\ggo_{\eta_i}^j))^K=0$, as they are pairwise non-equivalent as $K$-representations.  Note that the closedness of $\omega_2$ is a consequence of \eqref{ptpt}, concluding the proof.  
\end{proof}

We consider the Casimir operator of the $\kg$-representation $\pg_i$ relative to the bi-invariant inner product $\ip=-\kil_{\ggo}|_{\kg\times\kg} = \frac{1}{B_{2s}}g_b|_{\kg\times\kg}$ on $\kg$, given by 
$$
\cas_{\pg_i,\ip} = -\sum_{\alpha=1}^{\dim{\pg_i}} (\ad{Z^\alpha}|_{\pg_i})^2, \qquad i=1,\dots,s,
$$
and the Casimir operator $\cas_{\ad{\kg},\ip}$ of the adjoint representation.  Recall that $\pg_i$ is equivalent to the isotropy representation $\qg_i$ of the homogeneous space $G_i/\pi_i(K)$ (see \eqref{red6}).  The following algebraic invariants of $G/K$ will appear in the computations below:
\begin{align}
Cas_i:= & \tr{\cas_{\pg_i,\ip}} = B_{2s}\sum\limits_{\alpha,\beta,\gamma} (c_{i\alpha,i\beta}^{2s\gamma})^2, \qquad 1\leq i\leq s, \label{casi-def} \\
Cas_0 :=& \tr{\cas_{\ad{\kg},\ip}}=\sum\limits_{\alpha,\beta,\gamma} (\overline{c}_{\alpha\beta}^\gamma)^2 = \lambda_1\dim{\kg_1}+\dots+\lambda_t\dim{\kg_t}. \notag
\end{align}
We are using that $\kil_{\kg_j}=\lambda_j\kil_{\ggo}|_{\kg\times\kg}$ and that $\tr{\cas_{\ad{\kg_j},-\kil_{\kg_j}}} = \dim{\kg_j}$ for any $j=1,\dots,t$ in the last equality.  On the other hand, it follows from \cite[Lemma 3.1, (iii)]{stab} that 
$$
\tr{\cas_{\qg_i,-\kil_{\ggo_i}|_{\pi_i(\kg)}}}+\tr{\cas_{\ad{\pi_i(\kg)},-\kil_{\ggo_i}|_{\pi_i(\kg)}}}=\dim{\pi_i(\kg)}=\dim{\kg},
$$
and since $\ip=c_i(-\kil_{\ggo_i})(\pi_i\cdot,\pi_i\cdot)$ and consequently,  
$$
\cas_{\pg_i,\ip}=\cas_{\qg_i,c_i(-\kil_{\ggo_i})|_{\pi_i(\kg)}}=\frac{1}{c_i}\cas_{\qg_i,-\kil_{\ggo_i}|_{\pi_i(\kg)}}, \qquad \cas_{\ad{\kg},\ip}=\frac{1}{c_i}\cas_{\ad{\pi_i(\kg)},-\kil_{\ggo_i}|_{\pi_i(\kg)}},
$$
we obtain that
\begin{equation}\label{casi-def-2}
Cas_i+Cas_0 = \frac{1}{c_i}\dim{\kg}, \qquad \forall i=1,\dots,s.
\end{equation}

We are now ready to perform the main computation of this paper.   

\begin{proposition}\label{ompt4}
Consider any $G$-invariant metric $g=(x_1,\dots,x_{2s-1})_{g_b}$ as in \eqref{metric}.  Then $g(H_Q,d\omega)=0$ for any $\omega\in(\Lambda^2\pg^*)^K$ such that $\omega(\overline{\pg},\cdot)=0$ if and only if for all $s+1\leq j<k\leq 2s-1$,
\begin{equation}\label{ompt5}
x_k\left(E_k+Cas_0F_k+2Cas_0\left(\tfrac{1}{x_{j}^2}-\tfrac{1}{x_{k}^2}\right)\right)C_j =  x_j\left(E_j+Cas_0F_j\right)C_k, 
\end{equation} 
where 
\begin{align*}
E_{s+j}:=&\tfrac{Cas_1}{x_1^2}+\dots+\tfrac{Cas_j}{x_j^2}+A_{s+j}\tfrac{Cas_{j+1}}{x_{j+1}^2}, \qquad\forall j=1,\dots,s-1, \\
F_k:=&\tfrac{1}{x_{s+1}^2}+\dots+\tfrac{1}{x_{k-1}^2}+\tfrac{D_k}{B_k}\tfrac{1}{x_k^2}, \qquad\forall k=s+1,\dots,2s-1. 
\end{align*}
\end{proposition}

\begin{proof}
We consider the $g_b$-orthonormal basis $\{ e_\alpha^i:1\leq i\leq 2s-1,\; 1\leq\alpha\leq\dim{\pg_i}\}$ of $\pg$ given in \eqref{basis} and \eqref{basis2}.  It follows from Lemma \ref{ompt3}, (iii) that $\omega$ can be assumed to be as in parts (i) and (ii) of the same lemma.  Recall from \eqref{Qbcc} that the $3$-form $\widetilde{Q}$ is $g_b$-coclosed.  It therefore follows from \eqref{HQ}, \eqref{Hdo1}, \eqref{pipj} and \eqref{getai} (which in particular implies that $Q(\kg,e_\alpha^i)=0$ for all $i\leq s$, $[e_\alpha^i,e_\beta^j]=0$ for all $i\ne j\leq s$ and $[e_\alpha^i,e_\beta^j]_\kg=0$ for all $s<i\ne j$) that 
\begin{align*}
& g_b(H_Q,d\omega) = g_b(d\alpha_Q,d\omega)\\
=&-3\sum_{\substack{\alpha,\beta,\gamma\\ i,j;\; s<k}} \tfrac{1}{x_ix_jx_k} \left(Q([e_\alpha^i,e_\beta^j]_\kg,e_\gamma^k) - Q([e_\alpha^i,e_\gamma^k]_\kg,e_\beta^j) + Q([e_\beta^j,e_\gamma^k]_\kg,e_\alpha^i)\right)\omega([e_\alpha^i,e_\beta^j]_{\widetilde{\pg}},e_\gamma^k) \\  
=& -3\sum_{\substack{\alpha,\beta,\gamma\\ i\leq s<j}} \tfrac{1}{x_i^2x_j}Q([e_\alpha^i,e_\beta^i]_\kg,e_\gamma^j)\omega([e_\alpha^i,e_\beta^i]_{\widetilde{\pg}},e_\gamma^j) \\
&-3\sum_{\substack{\alpha,\beta,\gamma\\ s<i}} \tfrac{1}{x_i^3} \left(Q([e_\alpha^i,e_\beta^i]_\kg,e_\gamma^i) - Q([e_\alpha^i,e_\gamma^i]_\kg,e_\beta^i) + Q([e_\beta^i,e_\gamma^i]_\kg,e_\alpha^i)\right)\omega([e_\alpha^i,e_\beta^i]_{\widetilde{\pg}},e_\gamma^i) \\ 
&-3\sum_{\substack{\alpha,\beta,\gamma\\ s<i\ne j}} \tfrac{1}{x_i^2x_j}Q([e_\alpha^i,e_\beta^i]_\kg,e_\gamma^j)\omega([e_\alpha^i,e_\beta^i]_{\widetilde{\pg}},e_\gamma^j) \\ 
&+3\sum_{\substack{\alpha,\beta,\gamma\\ s<i\ne j}} \tfrac{1}{x_i^2x_j}Q([e_\alpha^i,e_\gamma^i]_\kg,e_\beta^j)\omega([e_\alpha^i,e_\beta^j]_{\widetilde{\pg}},e_\gamma^i) 
-3\sum_{\substack{\alpha,\beta,\gamma\\ s<i\ne j}} \tfrac{1}{x_i^2x_j}Q([e_\beta^i,e_\gamma^i]_\kg,e_\alpha^j)\omega([e_\alpha^j,e_\beta^i]_{\widetilde{\pg}},e_\gamma^i) \\ 
=& -3\sum_{\substack{\alpha,\beta,\gamma\\ i\leq s<j}} \tfrac{1}{x_i^2x_j}Q([e_\alpha^i,e_\beta^i]_\kg,e_\gamma^j)\omega([e_\alpha^i,e_\beta^i]_{\widetilde{\pg}},e_\gamma^j) 
+6\sum_{\substack{\alpha,\beta,\gamma\\ s<i}} \tfrac{1}{x_i^3}  Q([e_\alpha^i,e_\gamma^i]_\kg,e_\beta^i)\omega([e_\alpha^i,e_\beta^i]_{\widetilde{\pg}},e_\gamma^i)\\ 
&-3\sum_{\substack{\alpha,\beta,\gamma\\ s<i, j}} \tfrac{1}{x_i^2x_j}Q([e_\alpha^i,e_\beta^i]_\kg,e_\gamma^j)\omega([e_\alpha^i,e_\beta^i]_{\widetilde{\pg}},e_\gamma^j)
+6\sum_{\substack{\alpha,\beta,\gamma\\ s<j< i}} \tfrac{1}{x_i^2x_j}Q([e_\alpha^i,e_\beta^i]_\kg,e_\gamma^j)\omega([e_\alpha^i,e_\gamma^j]_{\widetilde{\pg}},e_\beta^i).   
\end{align*}
The fourth line in the above computation consists of three summands, the first one is contained in the left summand of the last line and the other two form the right summand of the penultimate line.  We also note that the sixth line becomes the right summand of the last line (recall that $[e^i_\alpha,e^j_\beta]\in\pg_i$ and so $\omega([e^i_\alpha,e^j_\beta],e^i_\gamma)=0$ for any $s<i<j$).  

Using Lemma \ref{ompt2} and the fact that $\omega(e^i_\alpha,e^j_\beta)=0$ for all $\alpha\ne\beta$ (see Lemma \ref{ompt3}, (ii)), we write each $[e^i_\alpha,e^j_\beta]_\kg$ and each $[e^i_\alpha,e^j_\beta]_{\widetilde{\pg}}$ as a linear combination of the $g_b$-orthonormal bases $\{ e^{2s}_\delta\}$ and $\{ e^k_\delta\}$, respectively, to obtain that 
\begin{align*} 
g_b(H_Q,d\omega)=
&-3\sum_{\substack{\alpha,\beta,\gamma\\ i\leq s<j;\; i+s-1\leq k}} \tfrac{1}{x_i^2x_j}c_{i\alpha,i\beta}^{2s\gamma}Q(e_\gamma^{2s},e_\gamma^j) c_{i\alpha,i\beta}^{k\gamma}\omega(e_\gamma^k,e_\gamma^j) \\ 
&+6\sum_{\substack{\alpha,\beta,\gamma\\ s<i< k}} \tfrac{1}{x_i^3}c_{i\alpha,i\gamma}^{2s\beta}Q(e_\beta^{2s},e_\beta^i) c_{i\alpha,i\beta}^{k\gamma}\omega(e_\gamma^k,e_\gamma^i) \\
&-3\sum_{\substack{\alpha,\beta,\gamma\\ s<i, j;\, i\leq k}} \tfrac{1}{x_i^2x_j}c_{i\alpha,i\beta}^{2s\gamma}Q(e_\gamma^{2s},e_\gamma^j) c_{i\alpha,i\beta}^{k\gamma}\omega(e_\gamma^k,e_\gamma^j) \\
&+6\sum_{\substack{\alpha,\beta,\gamma\\ s<j<i}} \tfrac{1}{x_i^2x_j}c_{i\alpha,i\beta}^{2s\gamma}Q(e_\gamma^{2s},e_\gamma^j) c_{i\alpha,j\gamma}^{j\beta}\omega(e_\beta^j,e_\beta^i),
\end{align*}
which by Lemma \ref{ompt2}, (i) gives that ,  
\begin{align*} 
g_b(H_Q,d\omega) 
=& -3\sum_{\substack{\alpha,\beta,\gamma\\ i\leq s<j;\; i+s\leq k}} \tfrac{1}{x_i^2x_j}c_{i\alpha,i\beta}^{2s\gamma}Q(e_\gamma^{2s},e_\gamma^j) \tfrac{\sqrt{B_{2s}}}{\sqrt{B_k}}c_{i\alpha,i\beta}^{2s\gamma}\omega(e_\gamma^k,e_\gamma^j) \\ 
&-3\sum_{\substack{\alpha,\beta,\gamma\\ i\leq s<j}} \tfrac{1}{x_i^2x_j}c_{i\alpha,i\beta}^{2s\gamma}Q(e_\gamma^{2s},e_\gamma^j) \tfrac{A_{i+s-1}\sqrt{B_{2s}}}{\sqrt{B_{i+s-1}}}c_{i\alpha,i\beta}^{2s\gamma}\omega(e_\gamma^{i+s-1},e_\gamma^j) \\
%
%
&+6\sum_{\substack{\alpha,\beta,\gamma\\ s<i< k}} \tfrac{1}{x_i^3}\tfrac{1}{\sqrt{B_{2s}}}\overline{c}_{\alpha\gamma}^\beta Q(e_\beta^{2s},e_\beta^i) \tfrac{1}{\sqrt{B_k}}\overline{c}_{\alpha\beta}^\gamma\omega(e_\gamma^k,e_\gamma^i) \\
&-3\sum_{\substack{\alpha,\beta,\gamma\\ s<i, j;\; i<k}} \tfrac{1}{x_i^2x_j}\tfrac{1}{\sqrt{B_{2s}}}\overline{c}_{\alpha\beta}^\gamma Q(e_\gamma^{2s},e_\gamma^j) \tfrac{1}{\sqrt{B_k}}\overline{c}_{\alpha\beta}^\gamma\omega(e_\gamma^k,e_\gamma^j) \\
& -3\sum_{\substack{\alpha,\beta,\gamma\\ s<i, j}} \tfrac{1}{x_i^2x_j}\tfrac{1}{\sqrt{B_{2s}}}\overline{c}_{\alpha\beta}^\gamma Q(e_\gamma^{2s},e_\gamma^j) \tfrac{D_i}{\sqrt{B_i^3}} \overline{c}_{\alpha\beta}^\gamma\omega(e_\gamma^i,e_\gamma^j)\\ 
 & +6\sum_{\substack{\alpha,\beta,\gamma\\ s<j<i}} \tfrac{1}{x_i^2x_j}\tfrac{1}{\sqrt{B_{2s}}}\overline{c}_{\alpha\beta}^\gamma Q(e_\gamma^{2s},e_\gamma^j) \tfrac{-1}{\sqrt{B_i}}\overline{c}_{\beta\gamma}^\alpha\omega(e_\beta^j,e_\beta^i). 
 \end{align*}
 
We now use Lemma \ref{ompt2}, (ii) and Lemma \ref{ompt3}, (ii) to obtain that 
\begin{align*}
& g_b(H_Q,d\omega) \\
=& -3\sum_{\substack{\alpha,\beta,\gamma\\ i\leq s<j;\; i+s\leq k}} \tfrac{1}{x_i^2x_j}(c_{i\alpha,i\beta}^{2s\gamma})^2 \tfrac{-C_j}{\sqrt{B_{2s}B_j}} \tfrac{\sqrt{B_{2s}}}{\sqrt{B_k}} \tfrac{B_{2s}}{\sqrt{B_kB_j}}\omega_{kj}  \\
& -3\sum_{\substack{\alpha,\beta,\gamma\\ i\leq s<j}} \tfrac{1}{x_i^2x_j}(c_{i\alpha,i\beta}^{2s\gamma})^2 \tfrac{-C_j}{\sqrt{B_{2s}B_j}} \tfrac{A_{i+s-1}\sqrt{B_{2s}}}{\sqrt{B_{i+s-1}}} \tfrac{B_{2s}}{\sqrt{B_{i+s-1}B_j}}\omega_{(i+s-1)j} \\ 
&-6\sum_{\substack{\alpha,\beta,\gamma\\ s<i< k}} \tfrac{1}{x_i^3}\tfrac{1}{\sqrt{B_{2s}}} (\overline{c}_{\alpha\beta}^\gamma)^2 \tfrac{-C_i}{\sqrt{B_{2s}B_i}} \tfrac{1}{\sqrt{B_k}} \tfrac{B_{2s}}{\sqrt{B_kB_i}}\omega_{ki} \\
&-3\sum_{\substack{\alpha,\beta,\gamma\\ s<i, j;\, i<k}} \tfrac{1}{x_i^2x_j}\tfrac{1}{\sqrt{B_{2s}}}(\overline{c}_{\alpha\beta}^\gamma)^2 \tfrac{-C_j}{\sqrt{B_{2s}B_j}} \tfrac{1}{\sqrt{B_k}} \tfrac{B_{2s}}{\sqrt{B_kB_j}}\omega_{kj} \\ 
& -3\sum_{\substack{\alpha,\beta,\gamma\\ s<i, j}} \tfrac{1}{x_i^2x_j}\tfrac{1}{\sqrt{B_{2s}}} (\overline{c}_{\alpha\beta}^\gamma)^2 \tfrac{-C_j}{\sqrt{B_{2s}B_j}} \tfrac{D_i}{\sqrt{B_i^3}} \tfrac{B_{2s}}{\sqrt{B_iB_j}}\omega_{ij} \\
 & -6\sum_{\substack{\alpha,\beta,\gamma\\ s<j<i}} \tfrac{1}{x_i^2x_j}\tfrac{1}{\sqrt{B_{2s}}}(\overline{c}_{\alpha\beta}^\gamma)^2 \tfrac{-C_j}{\sqrt{B_{2s}B_j}} \tfrac{1}{\sqrt{B_i}} \tfrac{B_{2s}}{\sqrt{B_jB_i}}\omega_{ji} \\
%
%
=& 3\sum_{\substack{\alpha,\beta,\gamma\\ i\leq s<j;\; i+s\leq k}} \tfrac{1}{x_i^2x_j}B_{2s}(c_{i\alpha,i\beta}^{2s\gamma})^2 \tfrac{C_j}{B_jB_k} \omega_{kj}  
+3\sum_{\substack{\alpha,\beta,\gamma\\ i\leq s<j}} \tfrac{1}{x_i^2x_j}B_{2s}(c_{i\alpha,i\beta}^{2s\gamma})^2 \tfrac{C_jA_{i+s-1}}{B_jB_{i+s-1}} \omega_{(i+s-1)j} \\ 
& +3\sum_{\substack{\alpha,\beta,\gamma\\ s<i\ne j;\, i<k}} \tfrac{1}{x_i^2x_j}\tfrac{1}{B_kB_j}(\overline{c}_{\alpha\beta}^\gamma)^2 C_j \omega_{kj} 
+ 3\sum_{\substack{\alpha,\beta,\gamma\\ s<i\ne j}} \tfrac{1}{x_i^2x_j}\tfrac{D_i}{B_i^2B_j}  (\overline{c}_{\alpha\beta}^\gamma)^2 C_j \omega_{ij} \\
&+6\sum_{\substack{\alpha,\beta,\gamma\\ s<i< k}} \tfrac{1}{x_i^3} (\overline{c}_{\alpha\beta}^\gamma)^2 \tfrac{C_i}{B_kB_i}\omega_{ki} 
+6\sum_{\substack{\alpha,\beta,\gamma\\ s<j<i}} \tfrac{1}{x_i^2x_j}  (\overline{c}_{\alpha\beta}^\gamma)^2 \tfrac{C_j}{B_jB_i}\omega_{ji}. 
\end{align*}
The fourth line above goes to the left summand of the last line and the fifth and sixth lines form the penultimate line.  

It now follows from \eqref{casi-def} that  
\begin{align*}
& g_b(H_Q,d\omega) = \\
=& 3\sum_{\substack{s<j,k}} \left(\tfrac{Cas_1}{x_1^2}+\dots+\tfrac{Cas_{k-s}}{x_{k-s}^2}\right) \tfrac{C_j}{B_jB_kx_j} \omega_{kj}  
+3\sum_{\substack{s<j,k}} \tfrac{Cas_{k-s+1}}{x_{k-s+1}^2}A_k \tfrac{C_j}{B_jB_kx_j} \omega_{kj} \\ 
&+ 3Cas_0\left(\sum_{\substack{s<i\ne j;\, i<k}} \tfrac{1}{x_i^2}\tfrac{C_j}{B_kB_jx_j} \omega_{kj} 
+ \sum_{\substack{s<k\ne j}} \tfrac{1}{x_k^2}\tfrac{D_kC_j}{B_k^2B_jx_j} \omega_{kj}\right) \\ 
&+6Cas_0\sum_{\substack{s<j< k}} \tfrac{1}{x_j^3}  \tfrac{C_j}{B_kB_j}\omega_{kj} -6Cas_0\left(\sum_{\substack{s<j<k}} \tfrac{1}{x_k^2x_j} \tfrac{C_j}{B_jB_k}\omega_{kj}\right)\\ 
%
%
=& 3\sum_{\substack{s<k,j}} E_k\tfrac{C_j}{B_kB_jx_j} \omega_{kj} +3Cas_0\sum_{\substack{s<k,j}} F_k\tfrac{C_j}{B_kB_jx_j} \omega_{kj} 
+6Cas_0\sum_{\substack{s<j< k}} \left(\tfrac{1}{x_j^2}-\tfrac{1}{x_k^2}\right) \tfrac{C_j}{B_kB_jx_j}\omega_{kj}.
%
\end{align*}

This implies that $g_b(H_Q,d\omega)=0$ for any skew-symmetric matrix $[\omega_{kj}]$ if and only if 
$$
\left(E_k+Cas_0F_k\right)\tfrac{C_j}{x_j} + 2Cas_0\left(\tfrac{1}{x_j^2}-\tfrac{1}{x_k^2}\right) \tfrac{C_j}{x_j} 
= \left(E_j+Cas_0F_j\right)\tfrac{C_k}{x_k}, \qquad \forall j<k,
$$
concluding the proof.
\end{proof}

We are finally in a position to state and prove our main result. 

\begin{theorem}\label{HQgb}
Let $M=G/K$ be a homogeneous space, where $G$ is compact, connected and semisimple and $K$ is connected, and fix decompositions $\ggo=\ggo_1\oplus\dots\oplus\ggo_s$ and $\kg=\zg(\kg)\oplus\kg_1\oplus\dots\oplus\kg_t$ in simple factors.  We assume that conditions (i) and (ii) below \eqref{h} hold and that $M=G/K$ is aligned with positive constants $c_1,\dots,c_s$ and $\lambda_1,\dots,\lambda_t$ (see Definition \ref{aligned}).  Consider any $G$-invariant metric of the form
$$
g=x_1g_b|_{\pg_1\times\pg_1}+\dots+x_{2s-1}g_b|_{\pg_{2s-1}\times\pg_{2s-1}}, \qquad x_1,\dots,x_{2s-1}>0, 
$$
where $g_b$ is the normal metric given by
$$
g_b=z_1(-\kil_{\ggo_1})+\dots+z_s(-\kil_{\ggo_s}), \qquad z_1,\dots,z_s>0,   
$$
and $\pg=\pg_1\oplus\dots\oplus\pg_{2s-1}$ is the $g_b$-orthogonal reductive decomposition of $G/K$ given by Proposition \ref{red-al}.  If we set 
\begin{align*}
a_j:=&\dim{\kg}\tfrac{1}{c_{j+1}x_{j+1}^2}+Cas_0\left(\tfrac{1}{x_{s+j}^2}-\tfrac{1}{x_{j+1}^2}\right), \qquad j=1,\dots,s-1,\\ 
b_j:=&\dim{\kg}\left(\tfrac{1}{c_1x_1^2}+\dots+\tfrac{1}{c_jx_j^2}\right) + 
Cas_0\left(\tfrac{1}{x_{s+1}^2}-\tfrac{1}{x_{1}^2}+\dots+\tfrac{1}{x_{s+j}^2}-\tfrac{1}{x_{j}^2}\right), 
\end{align*}
where $Cas_0 := \lambda_1\dim{\kg_1}+\dots+\lambda_t\dim{\kg_t}$, then $a_j,b_j\geq 0$ for all $j$ and the closed $3$-form $H_Q$ given in \eqref{HQ}, where
$$
Q=y_1\kil_{\ggo_1}+\dots+y_s\kil_{\ggo_s}, \qquad \tfrac{y_1}{c_1}+\dots+\tfrac{y_s}{c_s}=0, 
$$ 
is $g$-harmonic if and only if 
\begin{equation}\label{harm}
x_{s+k}\left(a_kA_k+b_k+2Cas_0\left(\tfrac{1}{x_{s+j}^2}-\tfrac{1}{x_{s+k}^2}\right)\right)C_j  
= x_{s+j}(a_jA_j+b_j)C_k, 
\end{equation}
for all $1\leq j<k\leq s-1$, where 
$$
A_j:=-\tfrac{c_{j+1}}{z_{j+1}}\left(\tfrac{z_1}{c_1}+\dots+\tfrac{z_j}{c_j}\right), \qquad C_j:=\tfrac{y_1}{c_1}+\dots+\tfrac{y_j}{c_j}+A_j\tfrac{y_{j+1}}{c_{j+1}}.  
$$
\end{theorem}

\begin{remark}\label{Hqgb-rem5}
Condition \eqref{harm} trivially holds if $s=2$ (cf.\ \S\ref{s2-sec}) or $\kg=0$ (cf.\ \S\ref{LG-sec}).  
\end{remark}

\begin{remark}\label{Hqgb-rem3}
In the case when $g=g_b$ is a normal metric (i.e., $x_1=\dots=x_{2s-1}=1$), the constants $a_j$ and $b_j$ become
$$
a_j=\dim{\kg}\tfrac{1}{c_{j+1}},  \qquad b_j=\dim{\kg}\left(\tfrac{1}{c_1}+\dots+\tfrac{1}{c_j}\right), \qquad \forall j=1,\dots,s-1,
$$
that is, they are just algebraic invariants of $G/K$, and condition \eqref{harm} becomes
$$
x_{s+k}(a_kA_k+b_k)C_j  
= x_{s+j}(a_jA_j+b_j)C_k, \qquad\forall 1\leq j<k\leq s-1.
$$ 
\end{remark}

\begin{remark}\label{Hqgb-rem2}
Using \eqref{casi-def-2}, we can rewrite the non-negative constants $a_j$ and $b_j$ as
\begin{align*}
a_j=&\tfrac{Cas_{j+1}}{x_{j+1}^2}+\tfrac{Cas_0}{x_{s+j}^2}, \qquad j=1,\dots,s-1,\\ 
b_j=&\tfrac{Cas_1}{x_1^2}+\dots+\tfrac{Cas_j}{x_j^2} + 
Cas_0\left(\tfrac{1}{x_{s+1}^2}+\dots+\tfrac{1}{x_{s+j}^2}\right), 
\end{align*}
where $Cas_i:= \frac{1}{c_i}\tr{\cas_{\qg_i,-\kil_{\ggo_i}|_{\pi_i(\kg)}}}$ and $\cas_{\qg_i,-\kil_{\ggo_i}|_{\pi_i(\kg)}}$ is the Casimir operator of the isotropy representation $\qg_i$ of the homogeneous space $G_i/\pi_i(K)$ (see \eqref{casi-def}).  
\end{remark}

\begin{remark}\label{Hqgb-rem6}
Since $Cas_0$ is the trace of the Casimir operator of the adjoint representation of $K$, we have that $Cas_0=0$ if and only if $\kg$ is abelian.  We also note that $Cas_i=0$ if and only if $\qg_i=0$ (i.e., $\pi_i(\kg)=\ggo_i$) for $i=1,\dots,s$; in particular, if at least one is zero then $\kg$ must be simple.  In the case when $Cas_i=0$ for all $i$, one obtains the Ledger-Obata space $M=K\times\dots\times K/\Delta K$.      
\end{remark}

\begin{remark}\label{Hqgb-rem}
The tuple $(A_1,\dots,A_{s-1})$ determines the metric $g_b$ up to scaling and for any given $g_b$, the tuple $(C_1,\dots,C_{s-1})$ determines the closed $3$-form $H_Q$ (see \eqref{HQgb-1} below).  In particular, the standard metric $g_{\kil}$ (i.e., $z_1=\dots=z_s=1$) corresponds to $A_j=-\frac{b_j}{a_j}$ for all $j=1,\dots,s-1$ and so any $H_Q$ is $\gk$-harmonic.  
\end{remark}

\begin{remark}\label{Hqgb-rem4}
Let $\mca^G$ denote the manifold of all $G$-invariant metrics on $M=G/K$.  If all the spaces $G_i/\pi_i(K)$, $i=1,\dots,s$ are isotropy irreducible and $K$ is either simple or one-dimensional, then the subspaces $\pg_1,\dots,\pg_{2s-1}$ are all $\Ad(K)$-irreducible and so 
$$
\dim{\mca^G}\geq s+\tfrac{s(s-1)}{2}=\tfrac{s(s+1)}{2},
$$ 
where equality holds if and only if $\pg_1,\dots,\pg_{s}$ are pairwise inequivalent (recall that $\pg_{s+1}\simeq\dots\simeq\pg_{2s-1}\simeq\kg$ and that the first $s$ $\pg_i$'s are pairwise inequivalent to the last $s-1$ ones by assumption).  Since each $g_b$ provides the submanifold $\{(x_1,\dots,x_{2s-1})_{g_b}\}$ of metrics covered by Theorem \ref{HQgb}, the subset $\tca$ of all $G$-invariant metrics covered by the theorem can be described as the union over the $s$-parametric space of all $g_b$'s of $(2s-1)$-dimensional submanifolds.  We do not know whether $\tca$ is equal to $\mca^G$ or not in the case when $\dim{\mca^G}=\tfrac{s(s+1)}{2}$.  
\end{remark}

\begin{proof}
It follows from Propositions \ref{ompt1}, \ref{ompt4} and \eqref{h} that $H_Q$ is $g$-harmonic if and only if condition \eqref{ompt5} holds.  On the other hand, it is easy to see that $\tfrac{D_{s+j}}{B_{s+j}} = 1+A_j$, which implies by using Remark \ref{Hqgb-rem2} that 
$$
E_{s+j}+Cas_0F_{s+j}=a_jA_j+b_j, \qquad \forall 1\leq  j\leq s-1.
$$
Thus conditions \eqref{harm} and \eqref{ompt5} are equivalent, concluding the proof.  
\end{proof}

The following corollaries of Theorem \ref{HQgb} answer all the questions made at the beginning of \S\ref{harm-sec} on the harmonicity of the closed $3$-forms $H_Q$'s.  

\begin{corollary}\label{HQgb-cor}
Let $M=G/K$ be an aligned homogeneous space as in the above theorem.  
\begin{enumerate}[{\rm (i)}] 
\item The standard metric $g_{\kil}$ is the unique (up to scaling) normal metric on $G/K$ such that any closed $3$-form $H_Q$ is $g_{\kil}$-harmonic.  

\item For any normal metric $g_b\ne g_{\kil}$ on $G/K$, there is a unique (up to scaling) closed $3$-form $H_Q$ which is $g_b$-harmonic.  

\item For any nonzero closed $3$-form $H_Q$ there exists a one-parameter family (up to scaling) of normal metrics $g_b(t)\ne\gk$ on $G/K$ such that $H_Q$ is $g_b(t)$-harmonic for all $t$.  
\end{enumerate}
\end{corollary}

\begin{proof}
Consider the tuple $(C_1,\dots,C_{s-1})$ determined by $Q$ and $g_b$.  Since
\begin{equation}\label{HQgb-1}
A\left[\begin{matrix}
y_1/c_1\\ 
\vdots\\ 
y_s/c_s
\end{matrix}\right] = \left[\begin{matrix}
C_1\\ 
\vdots\\ 
C_{s-1} \\
0
\end{matrix}\right], \qquad \mbox{where} \quad
A:=\left[\begin{matrix}
1&A_1&0&\cdots&0\\ 
1&1&A_2&\cdots&0\\ 
\vdots&&\ddots&\ddots&\vdots\\ 
1&\cdots&1&1&A_{s-1}\\
1&\cdots&1&1&1
\end{matrix}\right], 
\end{equation}
is an invertible $s\times s$ matrix (note that $\det{A}=(1-A_1)\dots(1-A_{s-1})$), we obtain that, conversely, given the metric $g_b$, the tuple $(C_1,\dots,C_{s-1})$ determines $Q$ .  Part (i) therefore follows from \eqref{harm} and Remark \ref{Hqgb-rem}.  

We now prove part (ii).  Given $g_b\ne \gk$, we consider the indexes $k_1,\dots,k_r$ such that $A_{k_i}\ne-\frac{b_{k_i}}{a_{k_i}}$, so $1\leq r\leq s-1$ and $a_jA_j+b_j=0$ for all $j\ne k_1,\dots,k_r$.  Recall from Remark \ref{Hqgb-rem3} the simpler forms of condition \eqref{harm} and the constants $a_j$ and $b_j$ in this case.  Thus the $g_b$-harmonicity condition \eqref{harm} for a given $H_Q$ is equivalent to $C_j=0$ for all $j\ne k_1,\dots,k_r$ and 
\begin{equation}\label{HQgb-2}
(a_{k_1}A_{k_1}+b_{k_1})C_{k_i} = (a_{k_i}A_{k_i}+b_{k_i})C_{k_1}, \qquad\forall i=2,\dots,r.
\end{equation}    
This implies that $C_{k_2},\dots,C_{k_r}$ and consequently $Q$, are determined by one parameter $C_{k_1}$ and the metric $g_b$, that is, $Q$ is determined up to scaling by $g_b$.  
 
Finally, we prove part (iii).  Given  $(y_1,\dots,y_{s})$ such that $\sum \frac{y_i}{c_i}=0$, we need to find a negative solution $(A_1,\dots, A_{s-1})$ to the system
\begin{equation}\label{HQgb-3s}
(a_{j}A_{j}+b_{j})C_{i} = (a_{i}A_{i}+b_{i})C_{j}, \qquad i,j=1,\dots,s-1,
\end{equation}
where $C_i=C_i(A_i):=\frac{y_1}{c_1} + \dots + \frac{y_i}{c_i} + A_i \frac{y_{i+1}}{c_{i+1}}$ (see \eqref{harm}).  It is therefore enough to consider only the equations with indexes $k_1,\dots,k_r$ such that $C_{k_i}\left(- \frac{b_{k_i}}{a_{k_i}}\right)\ne 0$, as all the remaining equalities hold by setting $A_{j}:=-\frac{b_j}{a_j}$ if $C_j\left(-\frac{b_j}{a_j}\right)=0$.  It is easy to see that $r\geq 1$ unless $Q=0$.  If we assume that each unknown $A_{k_i}$ is sufficiently close to $-\frac{b_{k_i}}{a_{k_i}}$ in order to have $C_{k_i}(A_{k_i})\ne 0$, then system \eqref{HQgb-3s} is equivalent to  
\begin{equation}\label{HQgb-4}
f_i(A_{k_i}):=\tfrac{a_{k_i}A_{k_i}+b_{k_i}}{C_{k_i}(A_{k_i})}  = \tfrac{a_{k_1}A_{k_1}+b_{k_1}}{C_{k_1}(A_{k_1})}, \qquad i=1,\dots,r.
\end{equation}
Since each function $f_i$ satisfies that $f_i\left(-\frac{b_{k_i}}{a_{k_i}}\right)=0$ and  
$f'_i\left(-\frac{b_{k_i}}{a_{k_i}}\right)= \frac{a_{k_i}}{C_{k_i}\left(-\frac{b_{k_i}}{a_{k_i}}\right)} \ne 0$, for any sufficiently small value $t=\frac{a_{k_1}A_{k_1}+b_{k_1}}{C_{k_1}(A_{k_1})}$, there exist negative numbers $A_{k_2},\dots,A_{k_{r}}$ such that \eqref{HQgb-4} holds, concluding the proof. 
\end{proof}

\begin{corollary}\label{HQgb-cor2-kabel}
Let $M=G/K$ be an aligned homogeneous space as in the above theorem such that $s\geq 3$ and $\kg$ is abelian (i.e., $Cas_0=0$ and $Cas_1,\dots,Cas_s>0$) and let $g_b$ be any normal metric.  
\begin{enumerate}[{\rm (i)}]
\item The set of all metrics of the form $g=(x_1,\dots,x_{2s-1})_{g_b}$ such that every closed $3$-form $H_Q$ is $g$-harmonic can be described as follows: 
\begin{enumerate}[{\rm a)}]
\item The numbers $x_2,\dots,x_s$ are recursively determined by $x_1$ in the following way:   
\begin{equation}\label{xjm1}
(-A_j)\tfrac{Cas_{j+1}}{x_{j+1}^2} = \tfrac{Cas_1}{x_1^2}+\dots+\tfrac{Cas_j}{x_j^2}, \qquad\forall j=1,\dots,s-1.  
\end{equation}
\item The $s$ numbers $x_1,x_{s+1},\dots,x_{2s-1}$ are independent positive parameters.    
\end{enumerate}

\item For any other $G$-invariant metric $g=(x_1,\dots,x_{2s-1})_{g_b}$, there is a unique (up to scaling)  $3$-form $H_Q$ which is $g$-harmonic.  
\end{enumerate}
\end{corollary}

\begin{proof}
We observe that according to Theorem \ref{HQgb} and Remark \ref{Hqgb-rem2}, any $H_Q$ is $g$-coclosed for a metric $g=(x_1,\dots,x_{2s-1})_{g_b}$ if and only if $a_jA_j+b_j=0$ for any $j=1,\dots,s-1$, which is equivalent to condition \eqref{xjm1} and so part (i) follows.    

Since the metrics involved in part (ii) are precisely those for which there is at least one $j$ such that $a_jA_j+b_j\ne 0$, the proof of this part follows in much the same way as the proof of part (ii) of Corollary \ref{HQgb-cor}.  
\end{proof}

\begin{corollary}\label{HQgb-cor2}
Let $M=G/K$ be an aligned homogeneous space as in the above theorem such that $s\geq 3$, $Cas_1>0$ and $\kg$ is not abelian (i.e., $Cas_0>0$) and let $g_b$ be any normal metric.  Then the set of all metrics of the form $g=(x_1,\dots,x_{2s-1})_{g_b}$ such that every closed $3$-form $H_Q$ is $g$-harmonic depends on the three parameters $x_1,x_s,x_{s+1}$ as follows: 
\begin{enumerate}[{\rm (i)}]
\item 
\begin{enumerate}[{\rm a)}]
\item $x_{s+1}=x_{s+2}=\dots=x_{2s-2}$.  

\item The numbers $x_2,\dots,x_{s-1}$ are recursively determined by $x_1$ and $x_{s+1}$ in the following way: 
\begin{equation}\label{xjm2}
(-A_j)\tfrac{Cas_{j+1}}{x_{j+1}^2} = \tfrac{Cas_1}{x_1^2}+\dots+\tfrac{Cas_j}{x_j^2} + 
Cas_0(j+A_j)\tfrac{1}{x_{s+1}^2}, \qquad\forall j=1,\dots,s-2,
\end{equation}
which is positive if and only if $x_{s+1}>u$ for a certain positive number $u$ depending on $x_1$.  Note that only the $x_j$'s such that $Cas_j>0$ are involved in the above equations, if $Cas_{j}=0$ then $\pg_{j}=0$ and so $x_j$ does not appear in $g$.  

\item The number $x_{2s-1}$ is determined by $x_1$, $x_s$ and $x_{s+1}$ as follows:
\begin{equation}\label{xjm3}
\tfrac{1}{x_{2s-1}^2} Cas_0(1-A_{s-1}) = \tfrac{Cas_1}{x_1^2}+\dots+\tfrac{Cas_{s-1}}{x_{s-1}^2} 
+A_{s-1}\tfrac{Cas_s}{x_s^2} + sCas_0\tfrac{1}{x_{s+1}^2},
\end{equation}
which is positive if and only if $x_s>v$ for a certain positive number $v$ depending on $x_1$ and $x_{s+1}$.   
\end{enumerate}

\item For any other $G$-invariant metric $g=(x_1,\dots,x_{2s-1})_{g_b}$, there exists at least one $3$-form $H_Q$ which is $g$-harmonic. 
\end{enumerate}
\end{corollary}

\begin{proof}
According to \eqref{harm}, if any $H_Q$ is $g$-coclosed, then  
$$
a_kA_k+b_k+2Cas_0\left(\tfrac{1}{x_{s+1}^2}-\tfrac{1}{x_{s+k}^2}\right)=0, \qquad\forall 1<k,
$$
and $a_jA_j+b_j=0$ for any $j=1,\dots,s-2$.  This implies that $x_{s+1}=\dots=x_{2s-2}$ (i.e., part a)), part b) follows and the only remaining equation to imply \eqref{harm} is:
\begin{equation}\label{xjm3-form}
a_{s-1}A_{s-1}+b_{s-1}+2Cas_0\left(\tfrac{1}{x_{s+1}^2}-\tfrac{1}{x_{{2s-1}}^2}\right)=0,  
\end{equation}
from which part c) easily follows.  

Since the metrics involved in part (ii) are precisely those for which either there is at least one $1\leq j\leq s-2$ such that $a_jA_j+b_j\ne 0$ or the left hand side of \eqref{xjm3-form} in non-zero, the proof of this part follows in much the same way as the proof of existence in part (ii) of Corollary \ref{HQgb-cor}, the uniqueness may not hold. 
\end{proof}

\begin{example}\label{HQgb-cor3}
Let $M=G/K$ be an aligned homogeneous space as in the above theorem.  If $Cas_1=\dots=Cas_s=0$, then $M=K\times\dots\times K/\Delta K$ is a Ledger-Obata space ($K$ simple) and for any normal metric $g_b$, the unique metric $g=(x_{s+1},\dots,x_{2s-1})_{g_b}$ satisfying that all the closed $3$-forms $H_Q$ are $g$-harmonic is given, up to scaling, by
$$
g=\left(1,\dots,1,t\right)_{g_b}, \qquad \mbox{where}\quad t:=\sqrt{\tfrac{1-A_{s-1}}{s}}, \quad A_{s-1}=-\tfrac{z_1+\dots+z_{s-1}}{z_{s}}.  
$$
Indeed, it follows from Corollary \ref{HQgb-cor2} and Remark \ref{Hqgb-rem2} that $x_{s+1}=\dots=x_{2s-2}$ and 
$$
\tfrac{1}{x_{2s-1}^2} (1-A_{s-1}) =  s\tfrac{1}{x_{s+1}^2},
$$ 
and the formula for $A_{s-1}$ follows from the fact that $c_1=\dots=c_s=s$.  In the case when $g_b=\gk$, $A_{s-1}=1-s$ and so $t=1$ and $g=\gk$.  
\end{example}

\begin{example}\label{GLO2}
Consider the aligned homogeneous space $M=H\times\dots\times H/\Delta K$ as in Example \ref{GLO}, where $H$ is a simple Lie group and $K\subset H$ is a proper closed simple subgroup, together with a $\kil_\hg$-orthogonal reductive decomposition $\hg=\kg\oplus\qg$.  If $\kil_\kg=c\kil_\hg|_\kg$, then  
$$
c_1=\dots=c_s=s, \qquad Cas_0=\tfrac{c\dim{\kg}}{s}, \qquad Cas_1=\dots=Cas_s= \tfrac{(1-c)\dim{\kg}}{s}. 
$$
According to Corollary \ref{HQgb-cor2}, if we fix $\gk$ as the background metric (in particular, $A_{s-1}=1-s$), then every $H_Q$ is $g$-harmonic with respect to a given metric $g=(x_1,\dots,x_{2s-1})_{\gk}$ if and only if $x_{s+1}=\dots=x_{2s-2}$,  
$$
\tfrac{1}{x_{j+1}^2} = \tfrac{1}{j}\left(\tfrac{1}{x_1^2}+\dots+\tfrac{1}{x_j^2}\right),  \qquad \forall j=1,\dots,s-2,
$$
which implies that $x_1=\dots=x_{s-1}$, and so
$$
\tfrac{1}{x_{2s-1}^2}  = \tfrac{(1-c)(s-1)}{cs}\left(\tfrac{1}{x_{1}^2} 
-\tfrac{1}{x_s^2}\right) + \tfrac{1}{x_{s+1}^2}.
$$
\end{example}


\begin{thebibliography}{MM}
\bibitem[AF]{AgrFrr} {\sc I. Agricola, A. Ferreira}, Einstein manifolds with skew torsion, {\it Quart. J. Math.} {\bf 65} (2014), 717-741. 

\bibitem[AFF]{AgrFrrFrd} {\sc I. Agricola, A. Ferreira, T. Friedrich}, The classification of naturally reductive homogeneous spaces in dimensions $n\leq 6$, {\it Diff. Geom. Appl.} {\bf 39} (2015), 59-92.

\bibitem[BMS]{BlgMrnSmm} {\sc F. Belgun, A. Moroianu, U. Semmelmann}, Killing Forms on Symmetric Spaces, {\it Diff. Geom. Appl.} {\bf 24} (2006), 215-222.

\bibitem[B]{Brl} {\sc A. Borel}, Topology of Lie groups and characteristic classes, {\it Bull. Amer. Math. Soc.} {\bf 61} (1955), 397-432.  

\bibitem[Br]{Brd} {\sc G. Bredon}, Topology and Geometry, {\it GTM} {\bf 139} (1993), Springer.  


\bibitem[DZ]{DtrZll} {\sc J. D'Atri, W. Ziller}, Naturally reductive metrics and Einstein metrics on compact lie groups, {\it Mem. Amer. Math. Soc.} {\bf 215} (1979).

\bibitem[GS]{GrcStr} {\sc M. Garcia-Fernandez, J. Streets}, Generalized Ricci Flow, {\it AMS University Lecture Series} {\bf 76}, 2021.  

\bibitem[K]{Krs} {\sc A. Krishnan}, Diagonalizing the Ricci tensor, {\it J. Geom. Anal.} {\bf 31} (2021), 5638-5658.

\bibitem[LL]{stab} {\sc E.A. Lauret, J. Lauret}, The stability of standard homogeneous Einstein manifolds,  {\it Math. Z.}  {\bf 303}, 16 (2023).      
 
\bibitem[LW1]{nicebasis}  {\sc J. Lauret, C.E. Will}, On the diagonalization of the Ricci flow on Lie groups, {\it Proc. Amer. Math. Soc.} {\bf 141}, Number 10, (2013), 3651-3663.

\bibitem[LW2]{BRF}  {\sc J. Lauret, C.E. Will}, Bismut Ricci flat generalized metrics on compact homogeneous spaces, preprint 2023 (arXiv).  

\bibitem[NN]{NklNkn} {\sc Y.Y. Nikolayevsky, Y.G. Nikonorov}, On invariant Riemannian metrics on Ledger-Obata spaces, {\it Manusc. Math.} {\bf 158} (2019) 353-370.  

\bibitem[PR1]{PdsRff1} {\sc F. Podest\`a, A. Raffero}, Bismut Ricci flat manifolds with symmetries, {\it Proc. Royal Soc. Edinburgh: Sec. A, Math.}, in press.  

\bibitem[PR2]{PdsRff2} {\sc F. Podest\`a, A. Raffero}, Infinite families of homogeneous Bismut Ricci flat manifolds, {\it Comm. Contemp. Math.}, in press.  

\bibitem[S]{Smm} {\sc U. Semmelmann}, Conformal Killing forms on Riemannian manifolds, {\it Math. Z.} {\bf 245} (2003), 503-527.

\bibitem[V]{Vss} {\sc L.A. Visscher}, Cohomology of compact Lie groups, Bachelor thesis, Utrecht Univ. (2019).
\end{thebibliography}
\end{document}